\documentclass{C0Seidel}

\AtEndDocument{\bigskip{\footnotesize%
  \textsc{VH: Sorbonne Universit\'e and Universit\'e de Paris, CNRS, IMJ-PRG, F-75005 Paris, France \&
    Institut Universitaire de France. }\par
  \textit{E-mail address}, \texttt{vincent.humiliere@imj-prg.fr} \par
  \addvspace{\medskipamount}
  \textsc{AJ: Heidelberg University, Institut f{\"u}r Mathematik, Excellence Cluster STRUCTURES, 69120 Heidelberg, Germany }\par
  \textit{E-mail address}, \texttt{ajannaud@mathi.uni-heidelberg.de} \par 
  \addvspace{\medskipamount}
  \textsc{RL: Universit\'e Paris-Saclay, CNRS, Laboratoire de mathématiques d'Orsay, 91405, Orsay, France.} \par
  \textit{E-mail address}, \texttt{remi.leclercq@universite-paris-saclay.fr} \par
}}

\def\wt#1{\widetilde{#1}}
\def\wh#1{\widehat{#1}}
\def\m#1{\mathcal{#1}}

\renewcommand{\leq}{\leqslant}
\renewcommand{\geq}{\geqslant}

\def\bb#1{\mathbb{#1}} 

\newcommand{\Z}{\mathbb{Z}}
\newcommand{\R}{\mathbb{R}}

\newcommand{\F}{\mathbb{F}}

\newcommand{\Q}{\mathbb{Q}}

\newcommand{\CP}{\mathbb{C}\mathrm{P}}

\def\Symp{\mathrm{Symp}}

\def\id{\mathrm{id}}
\def\Ham{\mathrm{Ham}}

\def\Hamtilde{\widetilde{\mathrm{Ham}}}
\def\Hambar{\overline{\mathrm{Ham}}}
\def\Hamhat{\widehat{\mathrm{Ham}}}

\newcommand{\QH}{\mathrm{QH}}
\newcommand{\HH}{\mathrm{H}}
\newcommand{\CF}{\mathrm{CF}}
\newcommand{\HF}{\mathrm{HF}}

\def\FP{\star}
\def\PSS{\mathrm{PSS}}
\def\wrt{\!:\!}

\newcommand{\eps}{\varepsilon}

\newcommand{\diagram}{\xymatrix} 

\begin{document}

\title{Essential loops {in completions of Hamiltonian groups}}
\date{\today}
\author{Vincent Humili\`ere, Alexandre Jannaud, R\'emi Leclercq}

\maketitle
    
\begin{abstract} { 
  We initiate the study of the fundamental group of natural completions of the group of Hamiltonian diffeomorphisms, namely its $C^0$-closure $\Hambar(M,\omega)$ and its completion with respect to the spectral norm $\Hamhat(M,\omega)$. We prove that in some situations, namely complex projective spaces and rational Hirzebruch surfaces, certain Hamiltonian loops that were known to be non-trivial in $\pi_1\big(\Ham(M,\omega)\big)$ remain non-trivial in $\pi_1\big(\Hamhat(M,\omega)\big)$. This yields in particular cases, including $\CP^2$ and the monotone $S^2\times S^2$, the injectivity of the map $\pi_1\big(\Ham(M,\omega)\big)\to\pi_1\big(\Hamhat(M,\omega)\big)$ induced by the inclusion. The same results hold for the Hofer completion of $\Ham(M,\omega)$. Moreover, whenever the spectral norm is known to be $C^0$-continuous, they also hold for $\Hambar(M,\omega)$.

  Our method relies on computations of the valuation of Seidel elements and hence of the spectral norm on $\pi_1\big(\Ham(M,\omega)\big)$. Some of these computations were known before, but we also present new ones which might be of independent interest.
  For example, we show that the spectral pseudo-norm is degenerate when $(M,\omega)$ is any non-monotone $S^{2}\times S^{2}$. At the contrary, it is a genuine norm when $M$ is the 1-point blow-up of $\CP^{2}$; it is unbounded for small sizes of the blow-up and become bounded starting at the monotone one.} 
\end{abstract}

\section{Introduction}\label{sec:intro}

Let $(M,\omega)$ be a closed symplectic manifold. We denote by $\Ham(M,\omega)$ its group of Hamiltonian diffeomorphisms. This group admits various completions that have been studied in the litterature. In particular, we may first consider its closure $\Hambar(M,\omega)$ with respect to the $C^0$-topology in the set of all homeomorphisms of $M$.  The elements of $\Hambar(M,\omega)$ are called \emph{Hamiltonian homeomorphisms}. Their behavior is quite well understood on surfaces, in particular thanks to Le Calvez's foliation techniques established in \cite{MR2217051}, see e.g. Le Calvez's work on the subject starting from \cite{MR2275671,MR2219272}.
 In higher dimension, partial results were obtained very recently by Buhovsky, Seyfaddini, and the first author, see e.g. \cite{MR3827210,MR4263685}.  We may also consider the completion of $\Ham(M,\omega)$ with respect to the so-called \emph{spectral norm}, which we will denote by $\Hamhat(M,\omega)$. Its study was initiated in \cite{Humiliere-completion}, and was more recently further developped in \cite{viterbo2022supports}. It has applications to Hamilton--Jacobi equations \cite{Humiliere-completion}, Symplectic Homogenization theory \cite{Viterbo-Homogenization}, and to conformally symplectic dynamics \cite{AHV}.

In this note, we are interested in the natural maps
\[\iota:\Ham(M,\omega)\to \Hambar(M,\omega),\quad j:\Ham(M,\omega)\to \Hamhat (M,\omega) \]
starting at the level of fundamental groups, on some symplectic manifolds of arbitrary dimension.
Indeed, while the homotopy type of $\Ham(M,\omega)$ has been \emph{extensively} studied, see below for some references which are used here, absolutely nothing is known about $\Hamhat(M,\omega)$ nor $\Hambar(M,\omega)$ beyond the case of surfaces where the map $\iota$ is known to be a homotopy equivalence. 
Before getting to the heart of the matter, let us point out that the analogous map $\Symp(M,\omega)\to \overline\Symp(M,\omega)$, between the groups of symplectic diffeomorphisms and homeomorphisms, was studied at the $\pi_0$ level by the second author \cite{Jannaud1, Jannaud2}.

\subsection{Main results}

The upshot of this work is a method which detects non-trivial elements in the image of $j_{*}$, i.e. non-trivial elements in $\pi_{1}\big(\Ham(M,\omega)\big)$ which survive in the fundamental group after taking its completion with respect to the spectral norm $\gamma$. 
When $\gamma$ is known to be $C^0$ continuous, this also yields essential loops of Hamiltonian homeomorphims.

For our method to work, all symplectic manifolds will be required to be \emph{rational}, meaning that their group of periods $\langle \omega, \pi_{2}(M) \rangle$ is generated by a unique positive element, which will be denoted by $\Omega$. In other words, $\langle \omega, \pi_{2}(M) \rangle = \Omega\bb Z$.

\medskip
The first application of our method concerns Hamiltonian circle actions and relies on deep work by McDuff and Tolman \cite{MR2210662}.  Recall that a fixed point component of a circle action is \emph{semifree} if it admits a neighborhood in which the stabilizer of every point is either trivial or the whole circle.

\begin{theo}\label{theo:general-circle-action}
  Consider a non-trivial Hamiltonian circle action $\Lambda$ on a compact, rational symplectic manifold $(M, \omega)$. Assume that its extremal fixed point components are semifree. Then $j_{*}([\Lambda])$ is non-trivial in  $\pi_{1}\big(\Hamhat(M,\omega)\big)$.
\end{theo}

As a direct consequence, we deduce the injectivity of $j_{*}$ in two different specific situations.
\begin{corol}\label{corol:S2S2-CP2}
    The map $j_{*} : \pi_{1}\big( \Ham(M,\omega) \big) \to \pi_{1}\big( \Hamhat(M,\omega) \big)$ is injective when $(M,\omega)$ is the monotone product $S^{2}\times S^{2}$ and for $\CP^{2}$ endowed with the Fubini--Study symplectic form.
\end{corol}

This consequence is straightforward since, in these two cases, all non-trivial elements of $\pi_1\big(\Ham(M,\omega)\big)$ may be represented by a Hamiltonian circle action satisfying the conditions of Theorem \ref{theo:general-circle-action}.
Whenever $\pi_1\big(\Ham(M,\omega)\big)$ has an element which does not admit such a representative, our method needs some extra work to produce an essential loop in the completion.
Notice in particular that semifreeness is not preserved under taking non-trivial powers. Therefore, we cannot directly extract information from Theorem \ref{theo:general-circle-action} about elements of the form $h^k$ with $|k|\neq 1$, even when $h$ can be represented by a circle action with semifree extremal components.

\begin{example}\label{example:non-monotone-S2xS2}
The fundamental group of the Hamiltonian diffeomorphism group of $S^{2}\times S^{2}$ endowed with a non-monotone product symplectic form is generated by two order-2 elements and the class of a loop $\Lambda$ of infinite order. All three are ensured to survive in {$\pi_{1}\big( \Hamhat(M,\omega) \big)$} by Theorem \ref{theo:general-circle-action}, however our method does not detect the even powers of $\Lambda$ so that, as far as we know, $j_{*}([\Lambda])$ might as well be of order 2 (see computations in Section \ref{sec:computations-Hirz-Surfaces}).
\end{example}

\medskip
However, we also have results in these more interesting situations, namely for complex projective spaces of any dimension and for all rational 1-point blow-ups of $\CP^{2}$.
 First, recall that Seidel \cite{MR1487754} proved that the group $\pi_1\big(\Ham(\CP^n,\omega_\mathrm{FS})\big)$ admits a non-trivial element $h_{n}$ of order $n+1$.
\begin{corol}\label{theo:CPn}
  For any $n\geq 1$, the element $j_*(h_{n})$ has order $n+1$ in  the group $\pi_1\big(\Hamhat(\CP^n,\omega_\mathrm{FS})\big)$.
\end{corol}

Second, recall that the symplectic $1$-point blow-ups $\bb F^{\mu} = (\CP^{2}\#\overline{\CP}^{2}, \omega'_{\mu})$ of $\CP^{2}$ admit a standard symplectic form $\omega'_{\mu}$ parameterized by a positive real number $\mu$. The fundamental group of their respective Hamiltonian diffeomorphism groups was computed by Abreu and McDuff \cite{MR1775741}: it is generated by a unique Hamiltonian circle action of infinite order.
\begin{theo}
  \label{theo:Hirzebruch-surfaces}
  The map $j_{*}$ is injective on all rational $1$-point blow-ups of $\CP^{2}$. In other words, whenever $\mu \in \bb Q$, the map $j_{*} : \pi_{1}\big( \Ham(\bb F^{\mu}) \big) \to \pi_{1}\big( \Hamhat(\bb F^{\mu}) \big)$ is injective.
\end{theo}

\begin{remark}\label{rem:completions} The above results also hold if one replaces the spectral distance $\gamma$ with Hofer's distance (which we denote here by $\delta$). Indeed, the Hofer continuity of spectral invariants yields $\gamma \leq \delta$ and thus a natural map $\Hamhat^{\delta}\!(M,\omega)\to\Hamhat^{\gamma}\!(M,\omega)$ between the respective completions, and hence a factorization of $j_*$
\[\pi_1\big(\Ham(M,\omega)\big)\to\pi_1\big(\Hamhat^{\delta}\!(M,\omega)\big)\to\pi_1\big(\Hamhat^{\gamma}\!(M,\omega)\big) \,.\]
Therefore, any class in $\pi_1\big(\Ham(M,\omega)\big)$  which is non-trivial in $\pi_1\big(\Hamhat^{\gamma}\!(M,\omega)\big)$ is also  non-trivial in $\pi_1\big(\Hamhat^{\delta}\!(M,\omega)\big)$.
\end{remark}

\subsection{Essential loops of Hamiltonian homeomorphisms in $\CP^n$.}

Remark that when the spectral norm is $C^{0}$ continuous, $j_{*}$ factorizes as
\begin{equation}\label{eq:inclusion2Hamhat-factor-through-Hambar}
    \begin{tikzcd}
    \pi_1\big(\Ham(M,\omega)\big) \ar[r, "\iota_{*}"] & \pi_1\big(\Hambar(M,\omega)\big)  \ar[r] & \pi_1\big(\Hamhat(M,\omega)\big) 
  \end{tikzcd}
\end{equation}
which ensures that elements  with non-trivial image via $j_*$ also have non-trivial image via $i_{*}$.

This is in particular the case in the complex projective space $\CP^n$, by a result of Shelukhin \cite{MR4480149}. Thus, from Corollaries \ref{corol:S2S2-CP2} and \ref{theo:CPn}, we deduce the following.

\begin{corol}\label{theo:CPn-homeo}
  For any $n\geq 1$, the element $\iota_*(h_{n})$ has order $n+1$ in  the group $\pi_1\big(\Hambar(\CP^n,\omega_\mathrm{FS})\big)$.
  {For $n=2$, the map $\iota_*:\pi_1\big(\Ham(\CP^2,\omega_\mathrm{FS})\big)\to\pi_1\big(\Hambar(\CP^2,\omega_\mathrm{FS})\big)$ is injective.}
\end{corol}

Note that the above corollary can also be obtained by non-symplectic methods, as was pointed out to us by Randal-Williams. Indeed, Sasao shows in \cite{MR0346783} that the action of $U(n+1)$ on $\CP^{n}$ induces an isomorphism between $\Z/(n+1)\Z$ and the fundamental group of the group of degree-1 continuous maps from $\CP^{n}$ to itself. This immediately implies that the class $h_n$, which is induced by the aforementioned action, is non-trivial in $\Ham(\CP^{n},\omega_{\mathrm{FS}})$ and in its closure $\Hambar(\CP^{n},\omega_{\mathrm{FS}})$. 

\begin{remark} Conjecturally, the spectral norm $\gamma$ is $C^0$ continuous for a large class of symplectic manifolds. If this holds in the settings of Theorems \ref{theo:general-circle-action} and \ref{theo:Hirzebruch-surfaces}, both statements admit analogs for the group of Hamiltonian homeomorphisms. 
  This follows from the above argument. 
\end{remark}

\subsection{The method}
\label{sec:method-1}

Our method is based on the following proposition of independent interest. In order to state it, we need to briefly recall a few well-established notions (necessary preliminaries are given in Section \ref{sec:prelim}). The quantum homology of $(M,\omega)$ is denoted by $QH(M, \omega)$. We let $\nu : QH(M,\omega)\to \Omega\Z\cup\{-\infty\}$ be the quantum valuation map, and $\mathcal S:\pi_1(\Ham(M,\omega))\to QH(M,\omega)^\times$ be the Seidel morphism. We consider the map $\Gamma : \pi_{1}\big(\Ham(M,\omega)\big) \to  \Omega\Z$ defined by $\Gamma(h)=\nu\big(\m S(h)\big)+\nu\big(\m S(h^{-1})\big)$.

\begin{prop}\label{prop:extension-Seidel} Assume that $(M,\omega)$ is a rational symplectic manifold, 
    then there exists a map $\wh\Gamma : \pi_1\big(\Hamhat(M,\omega)\big)\to \R$ so that the following diagram commutes:
\begin{equation}\label{eq:diagram-gamma-gammahat}
    \begin{tikzcd}
    \pi_1\big(\Ham(M,\omega)\big) \ar[d, "\Gamma"'] \ar[r, "j_{*}"] & \pi_1\big(\Hamhat(M,\omega)\big) \ar[ld, "\wh\Gamma"] \\
    \R  & 
  \end{tikzcd}
\end{equation}
\end{prop}

Based on this proposition, the proof of Theorem \ref{theo:general-circle-action}, Corollary \ref{theo:CPn}, and Theorem \ref{theo:Hirzebruch-surfaces} boils down to computing $\Gamma$ in order to show that it is non-zero on elements of $\pi_{1}\big( \Ham(M,\omega) \big)$. The diagram above then ensures that their image via $j_{*}$ in $\pi_1\big(\Hamhat(M,\omega)\big)$ cannot be trivial.

\begin{remark}
  McDuff already used the map $\Gamma$ to estimate the length of loops of Hamiltonian diffeomorphisms of 1-point blow-ups of $\CP^{2}$, see \cite[Lemma 5.1]{MR1959582}.
\end{remark}

\medskip
{An important ingredient in the} proof of Proposition \ref{prop:extension-Seidel} is the action of the Seidel homomorphism on spectral invariants. Indeed, it is not hard to see that, when $(M,\omega)$ is rational,  $\Gamma$ coincides with the restriction of the spectral pseudo-norm $\wt\gamma : \wt\Ham(M,\omega) \to \bb R$ to $\pi_{1}\big( \Ham(M,\omega) \big)$, see Section \ref{sec:seidel-morphism-1} for details.
 Proposition \ref{prop:extension-Seidel} is proved in Section \ref{sec:covering}.

\subsection{Computation of the spectral pseudo-norm}

As mentioned above, the proof of our main results boils down to computing the valuation of the Seidel elements associated to the elements of the fundamental group of Hamiltonian diffeomorphism groups.
These computations are based on work by Entov and Polterovich \cite{MR1979584} in the case of complex projective spaces, and on works by McDuff \cite{MR1959582}, by Ostrover \cite{ostrover}, and by Anjos and the third author \cite{MR3798327,MR3681107} in the case of the Hirzebruch surfaces ($S^{2} \times S^{2}$ and the 1-point blow-ups of $\CP^{2}$).

\medskip
Since the resulting function $\Gamma$ coincides with the restriction of the spectral pseudo-norm $\wt\gamma$ to $\pi_1\big(\Ham(M,\omega)\big)$ when the manifold is rational, we get explicit computations of the latter.
We collect below some phenomena of independent interest concerning $\wt\gamma$.
\begin{prop}\label{prop:facts-on-spectral-norm}
  Let $\wt\gamma : \pi_{1}\big(\Ham(M,\omega)\big) \to \bb R$ denote the restriction of the spectral pseudo-norm.
  \begin{enumerate}[label=\emph{(F\arabic*)}]
  \item Let $(M,\omega) = (S^{2}\times S^{2}, \omega_{\mu})$ with $\omega_{\mu}$ the product symplectic form with area $\mu$ on the first factor and $1$ on the second.

    When $\mu \in \bb Q$ and $\mu \neq 1$, $\wt\gamma$ is \emph{degenerate}: $\wt\gamma^{-1}(\{0\})    = \{ h^{2p} \,|\, p \in \bb Z \}$ where $h$ is the generator of infinite order of $\pi_{1}\big(\Ham(M,\omega)\big)$. 
  \item Let $(M,\omega)$ be any symplectic 1-point blow-up $(\CP^{2}\# \overline\CP^{2}, \omega'_{\mu})$ of $\CP^{2}$ or the $n$-dimensional complex projective space $(\CP^{n},\omega_{\mathrm{FS}})$, endowed with the standard Fubini--Study symplectic form.

    When $\mu \in \bb Q$, $\wt\gamma$ is \emph{non-degenerate}.
  \end{enumerate}
We now focus on 1-point blow-ups $(\CP^{2}\# \overline\CP^{2}, \omega'_{\mu})$ of $\CP^{2}$. The symplectic $\omega'_{\mu}$ has area $\mu>0$ on the exceptional divisor, and $1$ on the fiber\footnote{Recall that $\CP^{2}\# \overline\CP^{2}$ is the total space of the only non-trivial Hamiltonian fibration over $\CP^{1}$ with fiber $\CP^{1}$.}. Under these conventions, the only monotone such symplectic manifold is the one given by $\mu = \frac 12$.
\begin{enumerate}[label=\emph{(F\arabic*)}, resume]
\item The spectral norm $\wt\gamma$ is \emph{not bounded} on ``small'' rational 1-point blow-ups of $\CP^{2}$, for which $\mu < \frac 12$ and $\mu \in \bb Q$.
\item On the monotone 1-point blow-up, that is when $\mu = \frac 12$, we have $\wt\gamma(k) = 2$ for all non-trivial elements $k$ of $\pi_{1}\big(\Ham(\CP^{2}\# \overline\CP^{2}, \omega'_{\mathrm{mon}}) \big)$.
\item The image of the spectral norm is \emph{bounded} on ``big'' rational 1-point blow-ups of $\CP^{2}$, for which $\mu \geq \frac 12$ and $\mu \in \bb Q$.
\item  Let $h'$ be the generator of $\pi_{1}\big(\Ham(\CP^{2}\# \overline\CP^{2}, \omega'_{\mu}) \big)$ and $p\in\bb Z$, then the function $\mu \mapsto \Gamma(h'^{p})$, whose restriction to $\bb Q$ is $\mu \mapsto \wt\gamma(h'^{p})$, is continuous and piecewise linear on $\bb R$.
\end{enumerate}
\end{prop}

\begin{remark} Concerning (F1) above, it was proved in \cite{MR3681107} that Seidel's morphism is injective for \emph{all} Hirzebruch surfaces. The computations of Section \ref{sec:computations-Hirz-Surfaces} show that $\nu(\m S)$ is also injective. However, $\nu(\m S)$ does not factor through {$\pi_{1}\big(\Hamhat(M,\omega)\big)$} a priori and $\Gamma$ is \emph{not} injective. 

The fact (F2) for 1-point blow-ups of $\CP^{2}$ was proved in \cite{MR1959582}, by showing that $\Gamma$ is positive. In the second part of Section \ref{sec:computations-Hirz-Surfaces}, we compute the specific values of $\Gamma$ on the fundamental group of the Hamiltonian diffeomorphism group. These computations yield the facts (F2) to (F6).
\end{remark}

\paragraph{Acknowledgments.}

We thank Oscar Randal-Williams for useful email conversation and Sobhan Seyfaddini for pointing out an important mistake in a previous version of the paper. We also thank the anonymous referee for helpful comments.
The first and third authors are partially supported by the ANR grant 21-CE40-0002 (CoSy). The first author is also partially supported by the Institut Universitaire de France, and the ANR grant ANR-CE40-0014. The second author is supported by the Deutsche Forschungsgemeinschaft (DFG, German Research Foundation) under Germany’s Excellence Strategy EXC 2181/1 - 390900948 (the Heidelberg STRUCTURES Excellence Cluster). Some progress on this topic were made during a stay of the second author at the Laboratoire de Mathématiques d'Orsay which we thank for making that stay possible.

\section{Preliminaries}\label{sec:prelim}

This section is a quick recollection of the required constructions: quantum homology, Floer homology, the Seidel morphism, and spectral invariants.
We restrict ourselves to the case of strongly semi-positive symplectic manifolds, which includes all of our examples.

We would like to emphasize the fact that, even though all these constructions depend on the choice of an almost complex structure $J$, the isomorphism class of the resulting quantum and Floer homologies do not.
Ultimately, the spectral distance $\gamma$ (and hence the map $\Gamma$, based on the Seidel morphism), which is our main tool, do not depend on the choice of $J$.
In view of this, and to keep this presentation easily readable, we mostly removed almost complex stuctures from the discussion below.

\subsection{Quantum homology}
\label{sec:quantum-homology}

Let $\bb k$ be a field (which will be chosen to be $\Q$ in general, except for the case of $\CP^{n}$ in Section \ref{sec:proof-theor-CPn} for which $\bb k = \bb C$). The (small) \emph{quantum homology} of a strongly semi-positive symplectic manifold $(M,\omega)$ is the $\Z$-graded algebra defined as $\QH_{*}(M;\Lambda) = \HH_{*}(M;\bb k) \otimes_{\bb k} \Lambda$ where $\Lambda = \Lambda^{\mathrm{univ}}[q,q^{-1}]$ has coefficients in the ring of generalized Laurent series in the degree-0 variable $t$:
  \begin{equation*}
    \Lambda^{\mathrm{univ}} = \Big\{ \sum_{\kappa \in \R} r_{\!\kappa} t^{\kappa} \,\Big|\, r_{\!\kappa} \in \bb k \,\mbox{ s.t. } \forall c \in \R, \; \#\{\kappa > c \,|\, r_{\kappa} \neq 0 \} < \infty \Big\} 
  \end{equation*}
  and $q$ is a variable of degree 2.
  The grading of an element of the form $a\otimes q^{d}t^{\kappa}$ with $a \in \HH_{l}(M;\bb k)$ is simply given by $\deg(a\otimes q^{d}t^{\kappa}) = l + 2d$.

  The \emph{quantum intersection product} on $\QH_{*}(M;\Lambda)$ is a {deformation} of the usual intersection product on $\HH_{*}(M;\bb k)$ by counts of certain Gromov--Witten invariants.
  More precisely, for $a \in \HH_{k}(M;\bb k)$ and $b \in \HH_{l}(M;\bb k)$, 
  \begin{equation*}
    a * b = \sum_{B \in H_{2}^{S}(M;\Z)} (a * b)_{B} \otimes q^{-c_{1}(B)} t^{-\omega(B)}
  \end{equation*}
  where the sum runs over all spherical homology classes $B$, i.e. classes $B$ in the image $H_{2}^{S}(M;\Z)$ of the Hurewicz map $\pi_{2}(M) \to H_{2}(M;\Z)$.

 The class $(a * b)_{B} \in \HH_{*}(M;\bb k)$ has degree $k+l-\dim(M)+2c_{1}(B)$ and is defined by requiring its usual intersection product with any class $c \in \HH_{*}(M;\bb k)$ to be given by the Gromov--Witten invariant 
  \begin{equation*}
    (a * b)_{B} \cdot c = \mathrm{GW}^{M}_{B,3}(a,b,c) \in \bb k
  \end{equation*}
  which counts the number of spheres in $M$, in the class $B$, which meet cycles representing $a$, $b$ and $c$.
  The specific definition of $\mathrm{GW}^{M}_{B,3}$ is not necessary in this note. Only the following facts will be of interest:
  \begin{enumerate}[label=(\roman*)]
  \item As expected, $\deg(a*b) = \deg(a) + \deg(b) - \dim(M)$.
  \item The quantum intersection product turns the ring $\QH_{*}(M;\Lambda)$ into a $\Z$-graded commutative unital algebra.
  \item The unit of this algebra is the fundamental class $[M]$ of the symplectic manifold, seen as an element in $\QH_{2n}(M;\Lambda)$.
  \end{enumerate}

\paragraph{Quantum valuation}

The quantum homology algebra comes with a natural \emph{valuation} $\nu : \QH_{*}(M;\Lambda) \to \R \cup \{-\infty\}$ defined by
\begin{equation*}
  \label{eq:1}
  \nu\Big( \sum_{\kappa \in \R} a_{\kappa} \otimes q^{d_{\kappa}}t^{\kappa}\Big) = \max \{\kappa \,|\, a_{\kappa} \neq 0 \} \,.
\end{equation*}
Notice that any non-zero class has finite valuation because of the finiteness condition required in the definition of $\Lambda^{\mathrm{univ}}$. By convention, the zero class has valuation $-\infty$.

\paragraph{The monotone case}
    Morally speaking, the variables $q$ and $t$ respectively remember the first Chern number and the symplectic area of classes of spheres in $\pi_{2}(M)$. 
    
    In the (positively) monotone case, namely when
    \begin{equation*}
      \exists \lambda>0 \;\mbox{ such that }\;  \omega|_{\pi_{2}(M)} = \lambda \cdot c_{1}|_{\pi_{2}(M)} \,,
    \end{equation*}
    the morphisms $\omega$ and $c_{1}$ are linearly dependent on $\pi_{2}(M)$. Hence there is no need to carry around two variables: in this case, $\Lambda$ is usually replaced by $\bb k[[s]$ where $s$ is of degree $2N$, i.e. twice the \emph{minimal Chern number} $N$ of $M$ which is the positive generator of $\langle c_{1},\pi_{2}(M) \rangle = N\bb Z$.

    An element of the form $a \otimes s^{j}$ thus corresponds to $a \otimes q^{jN}t^{j\Omega}$ in the above description, with $\Omega$ the positive generator of $\langle \omega, \pi_{2}(M) \rangle$, which satisfies by monotonicity $\Omega = \lambda N$. The definition of the valuation has to be adapted consequently to $\nu\big( \sum_{j \in \Z} a_{j} \otimes s^{j}\big) = \Omega \cdot \max \{j \,|\, a_{j} \neq 0 \}$.

    This is for example the setting of \cite{MR1979584} whose results are used below.

\subsection{Floer homology}
\label{sec:floer-homol-viewp}

In order to prove Arnold's conjecture, Floer developed at the end of the 80's \cite{Floer_Morse_thry_Lagr_intersections,Floer_unregularized_grad_flow_symp_action,Floer_Symp_fixed_pts_holo_spheres,Floer_Witten_cx_inf_dim_Morse_thry} an infinite dimensional Morse--Bott-type homology for the symplectic action functional.
As Gromov--Witten invariants, this construction is another striking consequence of Gromov's celebrated work on pseudo-holomorphic curves \cite{MR0809718}. (And as Gromov--Witten invariants, it has had numerous applications.)

The upshot of Floer's construction, as far as this note is concerned, is that with a generic pair $(H,J)$ formed of a time-{periodic}, non-degenerate Hamiltonian function $H$ and a $\omega$-compatible almost complex structure $J$, one can define a $\Z$-graded complex $(\CF_{*}(M\wrt H),\partial_{(H,J)})$ whose homology $\HF_{*}(M,\omega)=\HH_{*}(\CF(M \wrt H),\partial_{(H,J)})$ satisfies the following properties.

\begin{enumerate}[label=(\roman*)]
\item There exist canonical \emph{continuation isomorphisms} between Floer homologies built from different admissible pairs of Floer data $(H,J)$ and $(H',J')$.
\item The \emph{pair-of-pants product} $\star$
turns Floer homology into a graded algebra, with unit $[M] \in \HF_{2n} (M,\omega)$, the fundamental class of $M$ seen as a Floer homology class.
\item There exist isomorphisms $\PSS : \QH_{*}(M;\Lambda) \to \HF_{*}(M,\omega)$ of graded commutative unital algebras.
\item Floer complexes are \emph{filtered}: any given $\alpha \in \R$ defines a subcomplex $(\CF_{*}^{\alpha}(M\wrt H),\partial_{(H,J)})$. Its homology does not depend on the choice of $J$ and is denoted by $\HF_{*}^{\alpha}(M\wrt H)$.
\end{enumerate}

\subsection{The Seidel morphism}
\label{sec:seidel-morphism}

The \emph{Seidel morphism} was described in \cite{MR1487754} in two quite different but equivalent ways.
On one hand, it can be seen as a morphism 
\begin{equation*}\label{eq:Seidel_geometric_absolute}
  \m S : \pi_{1}\big( \Ham(M,\omega) \big) \longrightarrow \QH_{*}(M;\Lambda)^{\times} \,, \qquad h \longmapsto \m S(h)
\end{equation*}
where $\QH_{*}(M;\Lambda)^{\times}$ denotes the multiplicative group of invertible elements of $\QH_{*}(M;\Lambda)$. The quantum class $S(h)$ is called the \emph{Seidel element} associated to $h$. It is defined by counting pseudo-holomorphic sections of a Hamiltonian fibration over $S^{2}$ with fibre $M$ obtained from a loop $(\phi^{t})_{t\in [0,1]} \in h$ via the clutching construction.

On the other hand, it can be seen as a \emph{representation}
\begin{equation*}\label{eq:Seidel_geometric_absolute}
  \m S : \pi_{1}\big( \Ham(M,\omega) \big) \longrightarrow \mathrm{Aut}\big(\HF_{*}(M,\omega)\big)  \,, \qquad h \longmapsto \m S_{h} \,.
\end{equation*}
The relation between these two viewpoints is straightforward: the automorphism of Floer homology $\m S_{h}$ is nothing but the pair-of-pants multiplication by $\m S(h)$ seen as a Floer homology class via the PSS morphism, i.e.
\begin{equation*}
  \forall b \in \HF_{*}(M,\omega) \,, \quad \m S_{h} (b) = \PSS(\m S(h)) \FP b \,.
\end{equation*}

\subsection{The spectral norm}
\label{sec:seidel-morphism-1}

The spectral norm is a norm defined on the universal cover of Hamiltonian diffeomorphism groups. It is based on the theory of \emph{spectral invariants}, introduced by Viterbo \cite{MR1157321} via generating functions, and adapted to the Floer-theoretic setting by Schwarz \cite{MR1755825} for symplectically aspherical manifolds and Oh \cite{MR2103018} for {more general symplectic manifolds}. 
Since then, they have been defined in a wide range of situations (and also had numerous deep consequences).

\paragraph{Spectral invariants}

Let $a\in \QH_*(M;\Lambda)$ be a non-zero quantum homology class. For any non-degenerate Hamiltonian function $H$ on $M$, the \emph{spectral invariant associated to $a$ with respect to $H$} is the real number
\begin{equation*}
  c(a\wrt H)=\inf\{\alpha\in\R\,|\, \mathrm{PSS}(a)\text{ belongs to the image of }\iota^\alpha\}
\end{equation*}
where $\iota^\alpha:HF_*^\alpha(M\wrt H)\to HF_*(M,\omega)$ is the map induced by the inclusion. These numbers enjoy a list of standard properties. First, they are ``spectral'' in the sense that $c(a\wrt H)$ belongs to the set of critical values of the action functional associated to $H$; this property which explains their name will not be needed here. They are also Hofer continuous, namely
\begin{equation*}
\int_0^1\!\min_{M}(H_t-K_t) \,dt\,\leq  c(a \wrt H)-c(a \wrt K)\leq \int_0^1\!\max_{M}(H_t-K_t) \,dt
\end{equation*}
for any two Hamiltonians $H(t,x)=H_t(x)$, $K(t,x)=K_t(x)$. In particular, $c(a\wrt H)$ continuously depends on $H$ and the map $c(a\wrt \,\cdot\,)$ extends continuously to all (possibly degenerate) Hamiltonians.

Recall that a smooth path of Hamiltonian diffeomorphisms  $(h^t)_{t\in[0,1]}$ with $h^0=\id$ is generated by a unique mean normalized Hamiltonian $H$, i.e. satisfying $\int_M H_t \, \omega^n=0$ for all $t$. Thus we may define spectral invariants of the isotopy $(h^t)$ by setting  $c(a\wrt(h^t))=c(a\wrt H)$. These invariants applied to isotopies have the nice feature that they are invariant by homotopy with fixed end points. Therefore, we obtain a well-defined map on the universal cover of $\Ham(M,\omega)$:
\[c(a\wrt\,\cdot\,):\wt\Ham(M,\omega)\to\R\,.\]

These maps satisfy the so-called \emph{triangle inequality}:
\begin{equation*}
c(a\ast b\wrt \wt h \wt k)\leq c(a\wrt\wt k)+c(b\wrt \wt h)
\end{equation*}
for any classes $a$, $b\in QH_*(M,\Lambda)$ such that $a\ast b\neq 0$, and any $\wt h$, $\wt k\in\wt\Ham(M,\omega)$. 

We will mainly use the spectral invariants associated to the fundamental class. Therefore, it will be convenient to introduce the notation $c_+(H)=c([M]\wrt H)$ and $c_+(\wt h)=c([M]\wrt\wt h)$. Note that the triangle inequality for $c_+$ takes the form
\begin{equation}\label{eq:triangle_ineq}
c_+(\wt h \wt k)\leq c_+(\wt h)+c_+(\wt k) \,.
\end{equation}

\paragraph{Relation between $c$, $\m S$, and $\nu$}

In the language of spectral invariants, the quantum valuation defined in Section~\ref{sec:quantum-homology} behaves as the spectral invariant computed with respect to the zero Hamiltonian / the identity in $\Hamtilde$. By this, we mean that for any non-zero quantum homology class $a \in \QH_{*}(M;\Lambda)$
\begin{equation}\label{eq:Valuation-as-SI-of-0}
   c(a\wrt\wt\id) = \nu(a) \,.
\end{equation}

 Moreover, because Seidel's representation provides automorphisms of \emph{filtered} Floer complexes, the action of any $h \in \pi_{1}\big(\Ham(M,\omega)\big)$ on $\wt\Ham(M,\omega)$ (or in other words the difference between lifts of $\phi \in \Ham(M,\omega)$ to the universal cover) yields, in terms of spectral invariants:
\begin{equation}\label{eq:Seidel-action-on-SI}
  \forall a\in \QH_{*}(M;\Lambda) \,, \quad c(a\wrt\wt k h) = c(\m S(h) * a\wrt \wt k)
\end{equation}
for all $\wt k \in \wt\Ham(M,\omega)$.
In the specific situation where $\wt k=\wt\id$ and $a = [M]$, we get
\begin{equation}\label{eq:cplus-equals-nu-of-Seidel}
    c_{+}(h) = c([M]\wrt h) = c(\m S(h) * [M]\wrt \wt\id) = \nu(\m S(h))
\end{equation}
respectively by definition of $c_{+}$, by (\ref{eq:Seidel-action-on-SI}), and finally by (\ref{eq:Valuation-as-SI-of-0}) and the fact that $[M]$ is the unit of the quantum homology algebra. This equality has been successfully used before, see e.g. Proposition 4.1 in \cite{MR1979584}, Section 6.3 of \cite{ostrover}, and Property (2.4) of spectral numbers in \cite{MR2563682}.

\paragraph{The spectral pseudo-norm}

The spectral pseudo-norm is defined as a group pseudo-norm $\wt\gamma : \wt\Ham(M,\omega) \to \R$ by 
$$\wt\gamma(\wt h) = c_{+}(\wt h) + c_{+}(\wt h^{-1}) = c_{+}(H) + c_{+}(\overline H)$$ 
with $H$ generating $(\phi_{H}^{t})_{t\in [0,1]} \in \wt h$ and $\overline H$ the Hamiltonian function defined by $\overline H_{t}(x) = -H_{1-t}(x)$. Standard computations show that $\overline H$ generates the isotopy $\phi_{H}^{1-t}(\phi_{H}^{1})^{-1} \in \wt h^{-1}$. The triangle inequality for $\wt\gamma$ as well as its non-negativity follow from (\ref{eq:triangle_ineq}) and the fact that $\wt\gamma(\wt\id)=0$.

Notice that (\ref{eq:cplus-equals-nu-of-Seidel}) yields for any $h\in\pi_1(\Ham(M,\omega))$
\begin{equation}
  \label{eq:gamma-equals-Gamma}
  \wt\gamma(h) = c_{+}(h) + c_{+}(h^{-1}) = \nu( \m S(h)) + \nu( \m S(h^{-1})) = \Gamma(h) 
\end{equation}
where $\Gamma$, defined by the last equality, is the central object of Proposition \ref{prop:extension-Seidel} from the introduction. Notice that this implies that
\begin{equation*}
   \wt\gamma(h)\in\langle\omega,\pi_2(M)\rangle, \quad \forall h\in\pi_1(\Ham(M,\omega)).
\end{equation*}

The pseudo-norm $\wt\gamma$ induces a genuine (non-degenerate) norm $\gamma$ on the group $\Ham(M,\omega)$, called the \emph{spectral norm}, by the formula (see \cite{MR2181090})
\begin{equation*}
  \gamma(\phi)=\inf\left\{\wt \gamma(\wt h) \,\big|\, \wt h\in\wt\Ham(M,\omega), \pi(\wt h)=\phi\right\}.
\end{equation*}

\paragraph{The $\gamma$-completion.}

The completion of $\Ham(M,\omega)$ with respect to the spectral norm is denoted by $\Hamhat(M,\omega)$. Recall that by definition it is the quotient of the set of $\gamma$-Cauchy sequences of elements in $\Ham(M,\omega)$ by the equivalence relation
\[(\phi_n)\sim(\psi_n) \quad \text{if and only if }\quad \gamma(\phi_n, \psi_n)\stackrel{n\to\infty}\longrightarrow 0.\]


\section{Proof of Proposition \ref{prop:extension-Seidel}}\label{sec:covering}

To be able to prove Proposition \ref{prop:extension-Seidel}, we first need to study the metric topologies induced by $\wt\gamma$ and $\gamma$. In this purpose, we prove a few preliminary lemmas and then prove Proposition \ref{prop:extension-Seidel} at the end of the section. Throughout the section, we assume that $(M,\omega)$ is a rational symplectic manifold and we let $\Omega$ denote the unique positive generator of $\langle \omega, \pi_2(M)\rangle$.

To circumvent the unpleasant fact that the pseudo-metric $\wt\gamma$ may be degenerate, we introduce
\[G=\left\{h\in \pi_1\big(\Ham(M,\omega)\big): \wt \gamma(h)=0\right\}.\]
We note that this is a normal subgroup of $\pi_1\big(\Ham(M,\omega)\big)$ and define $\mathcal H=\widetilde\Ham(M,\omega)/G$ the associated cover of $\Ham(M,\omega)$. The pseudo-metric $\wt \gamma$ descends to a non-degenerate metric on $\mathcal H$, still denoted by $\wt\gamma$. Indeed, for any $\wt k\in \widetilde\Ham(M,\omega)$, $h\in G$, by triangle inequality \[\wt\gamma(\wt k h)-\wt\gamma(h)\leq \wt\gamma(\wt k)\leq \wt\gamma(\wt k h)+\wt\gamma(h^{-1})\]
 which yields $\wt\gamma(\wt k h)= \wt\gamma(\wt k)$, since $\wt \gamma(h)=\wt\gamma(h^{-1})=0$.

Denoting by $\pi$ the canonical projection $\mathcal H\to\Ham(M,\omega)$, we still have
\begin{equation}
  \label{eq:gamma-gammatilde}
  \gamma(\phi)=\inf\left\{\wt \gamma(\wt h):\wt h\in\mathcal H, \pi(\wt h)=\phi\right\}.
\end{equation}

\begin{lemma}\label{lemma:isometry-local}
With respect to the spectral metrics $\wt \gamma$ and $\gamma$, the projection $\pi:\mathcal H\to\Ham(M,\omega)$ is a local isometry.
\end{lemma}

By \textit{local isometry}, we mean that for all $\wt h\in\mathcal H$, there exist open neighborhoods $U$ and $V$, respectively of $\wt h$ and $\pi(\wt h)$, such that $\pi$ restricts to a bijective isometry $\pi : U \to V$.

\begin{proof} Let $\wt \phi$, $\wt \psi$ in $\m H$, such that $\wt \gamma(\wt \phi,\wt \psi)<\frac\Omega3$. In order to show that $\pi$ is a local isometry, we first prove that $\gamma\big(\pi(\wt\phi), \pi(\wt\psi)\big)=\wt \gamma(\wt \phi,\wt \psi)$. By (\ref{eq:gamma-gammatilde}), the left hand side is smaller or equal than the right hand side. Let us prove the other inequality.

 Let $\rho:=\pi(\wt\psi^{-1}\wt\phi)=\pi(\wt\psi)^{-1}\circ\pi(\wt\phi)$ so that by definition $\gamma(\rho)=\gamma\big(\pi(\wt\phi), \pi(\wt\psi)\big)$. Let $\eps\in(0, \frac\Omega3)$. By (\ref{eq:gamma-gammatilde}), there exists $h\in \pi_1\big(\Ham(M,\omega)\big)/G$ such that
  \begin{equation}
    \label{eq:isom1}
    \gamma(\rho)\geq\wt\gamma(\wt\psi^{-1}\wt\phi \,h)-\eps.
  \end{equation}
The triangle inequality then yields
\begin{equation*}
  \label{eq:isom2}
  \gamma(\rho)\geq\wt\gamma(h)-\wt\gamma(\wt\psi^{-1}\wt\phi)-\eps>\wt\gamma(h)-\tfrac{\Omega}3-\eps.
\end{equation*}
Since $\gamma(\rho)\leq \wt \gamma(\wt \phi,\wt \psi)<\frac\Omega3$, we deduce $\wt\gamma(h)<2\frac\Omega3+\eps<\Omega$. Now, $\wt\gamma(h)\in\Omega\Z_{\geq0}$, thus $\wt\gamma(h)=0$ and $h$ is the trivial element in $\pi_1(\Ham(M,\omega))/G$.
Using (\ref{eq:isom1}) again, we obtain $\gamma(\rho)\geq \wt\gamma(\wt\psi^{-1}\wt\phi )-\eps$. Since this holds for arbitrarily small $\eps$, this concludes the proof that $\gamma(\rho)\geq \wt\gamma(\wt\phi, \wt\psi)$.

Hence $\pi$ preserves the distances on open balls of radius $\frac \Omega3$ so that, for all $\wt\phi$ and $\phi=\pi(\wt\phi)$, $\pi(B(\tilde\phi, \frac\Omega3))\subset B(\phi, \frac\Omega3)$
and it only remains to prove the opposite inclusion.
Let $\psi\in B(\phi, \frac\Omega3)$, by definition $\gamma(\phi^{-1}\psi)<\frac\Omega3$. By (\ref{eq:gamma-gammatilde}), $\phi^{-1}\psi$ admits a lift $b\in\m H$ with $\wt\gamma(b)<\frac\Omega3$. We set $\wt\psi:=b\,\wt\phi$; then, $\wt\gamma(\wt\phi, \wt\psi)<\frac\Omega3$ and $\pi(\wt\psi)=\psi$. This shows that $\psi$ admits a preimage by $\pi$ in $B(\wt\phi, \frac\Omega3)$, and concludes the proof that $\pi$ restricts to a bijective isometry $\pi : B(\wt\phi, \frac\Omega3) \to B(\phi, \frac\Omega3)$.
\end{proof}

By definition of $\gamma$, the map
$\pi$ is 1-Lipschitz, hence it extends to a continuous map $\wh\pi:\widehat{\mathcal H}\to \Hamhat(M,\omega)$ between  the spectral completions of $\mathcal H$ and $\Ham(M,\omega)$.

\begin{lemma}\label{lemma:isometry-local-completion}
  The map $\wh\pi$ is a local isometry, and it is surjective. As a consequence, it descends to a homeomorphism between  $\widehat{\mathcal H}/\ker(\wh\pi)$ and $\Hamhat(M,\omega)$.
\end{lemma}

\begin{proof}
  As a continuous extension of $\pi$, the map $\wh\pi$ also preserves distances. Let $\delta>0$ be such that $\pi$ restricts to a bijective isometry between the balls of radius $\delta>0$ centered at the identity.
  We need to check that $\wh\pi \big(B(\wt\phi,\delta)\big) = B(\phi,\delta)$, 
  for any $\phi\in \Hamhat(M,\omega)$ and any lift $\wt\phi\in\wh{\m H}$.

  Let $\psi\in B(\phi,\delta)$. Choose a Cauchy sequence $(\chi_n)$ which represents $\phi^{-1}\psi$. We may assume without loss of generality that $\gamma(\chi_n)<\delta$ for all $n>0$.
  Let $b_n$ be the unique preimage of $\chi_n$ under $\pi$ in the $\delta$-ball centered at the identity. Since $\pi$ is an isometry, $(b_n)$ is a Cauchy sequence whose class $b\in\wh{\m H}$ satisfies $\wt\gamma(b) = \gamma(\phi^{-1}\psi) <\delta$. Thus the element $\wt\psi:=b\,\wt\phi$ of $B(\wt\phi,\delta)$ lifts $\psi$.
  This proves $\wh\pi \big(B(\wt\phi,\delta)\big) \subset B(\phi,\delta)$. Since the opposite inclusion is obvious, $\wh\pi$ is a local isometry.

  The argument above proves that the image of $\wh\pi$ contains a $\delta$-neighborhood of $\Ham(M,\omega)$ which shows that $\wh\pi$ is surjective. As a consequence, $\wh\pi$ descends to a continuous bijection  $\widehat{\mathcal H}/\ker(\wh\pi)\to\Hamhat(M,\omega)$. Being a local homeomorphism,  $\wh\pi$ is open, hence the induced map is a homeomorphism.
\end{proof}

We now explicit the relation between $\pi$ and $\wh\pi$.

\begin{lemma}\label{lem:covering}
  The group homomorphisms $\pi$ and $\wh\pi$ satisfy $\ker(\wh\pi)=\ker(\pi)$, and we have a pull-back diagram of coverings 
\begin{equation}
  \label{eq:diagram1}
  \diagram{\mathcal H \ar[r]^{\wt j}\ar[d]_\pi & \widehat{\mathcal H}\ar[d]^{\wh\pi} \\ \Ham(M,\omega)\ar[r]^j & \Hamhat(M,\omega)}
\end{equation}
\end{lemma}

\begin{proof}
    The inclusion $\ker(\pi)\subset\ker(\wh\pi)$ is obvious. On the other hand, let $h\in\ker(\wh\pi)\subset \widehat{\mathcal H}$ and $(h_n)$ be a $\wt\gamma$-Cauchy sequence in $\mathcal H$ which represents $h$. By assumption, $\wh\pi(h)=\id$, which means $\gamma(\pi(h_n))\to 0$. By (\ref{eq:gamma-gammatilde}), there exist elements $y_n\in\ker(\pi)$ such that $\wt\gamma(h_ny_n^{-1})\to 0$. Since $(h_n)$ is a $\wt\gamma$-Cauchy sequence, so is $(y_n)$. Now $\ker(\pi)$ being discrete for the $\tilde\gamma$-topology by Lemma \ref{lemma:isometry-local}, 
  this implies that there exists some element $y\in\ker(\pi)$ such that $y_n=y$ for $n$ large enough. From $\wt\gamma(h_ny_n^{-1})\to 0$, we deduce $h_n\to y$, hence $h=y$. This shows that $h\in\ker(\pi)$, hence concludes the proof that $\ker(\pi)=\ker(\wh\pi)$.

  Recall the classical fact that if $H$ is a discrete normal subgroup of a Hausdorff topological group $G$, then the quotient map $G\to G/H$ is a covering map \cite[Cor. 3.7]{Paulin}. Since $\ker(\wh\pi)=\ker(\pi)$ is discrete, this applies in our case and shows that $\wh\pi:\widehat{\mathcal H}\to \Hamhat(M,\omega)=\widehat{\mathcal H}/\ker(\wh\pi)$ is a covering map.
\end{proof}

It only remains to show how Proposition \ref{prop:extension-Seidel} follows from Lemma \ref{lem:covering}.

\begin{proof}[Proof of Proposition \ref{prop:extension-Seidel}]
  The respective actions of $\pi_1\big(\Ham(M,\omega)\big)$ and $\pi_1\big(\Hamhat(M,\omega)\big)$ on the fibers of the covering maps $\pi$ and $\wh\pi$ over $\id$ are intertwined by Diagram (\ref{eq:diagram1}). In particular, denoting by $1\in\mathcal H$ and $\wh 1\in\wh{\mathcal H}$ the identity elements,  we have $\wt{j}(\alpha\cdot 1)=j_*(\alpha)\cdot \wh{1}$ for any $\alpha\in\pi_1\big(\Ham(M,\omega)\big)$. 

The action of $\pi_1\big(\Ham(M,\omega)\big)$ on $\ker(\pi)$ is induced by the group law on $\pi_1\big(\Ham(M,\omega)\big)$, thus we may write $\wt\gamma(\alpha\cdot 1)=\wt\gamma(\alpha)$. By~(\ref{eq:gamma-equals-Gamma}), we deduce
 \[\Gamma(\alpha)=\wt\gamma(\alpha\cdot 1), \quad \forall \alpha\in\pi_1\big(\Ham(M,\omega)\big).\]
We analogously define  \[\wh\Gamma(\alpha)=\wt\gamma(\alpha\cdot \wh1), \quad \forall \alpha\in\pi_1\big(\Hamhat(M,\omega)\big).\]
Then, we have for any $\alpha\in\pi_1\big(\Ham(M,\omega)\big)$
\[\wh\Gamma\circ j_\ast(\alpha)=\wt\gamma(j_\ast(\alpha)\cdot \wh1)=\wt\gamma(\wt{j}(\alpha\cdot 1))=\wt\gamma(\alpha\cdot 1)=\Gamma(\alpha),\]
where the third equality above follows from the obvious fact that $\wt j$ is an isometry for $\wt \gamma$.
This proves the proposition.
\end{proof}

\section{Computations}\label{sec:proofs}

\subsection{Proof of Theorem \ref{theo:general-circle-action}}
\label{sec:proof-theor-McDuff-Tolman}

This is a simple remark based on the following deep result from McDuff and Tolman which is part of \cite[Theorem 1.10]{MR2210662}.

\begin{theo*}[McDuff--Tolman]
  Let $\Lambda$ be a Hamiltonian circle action on a compact symplectic manifold $(M,\omega)$, generated by the moment map $K : M \to \R$. Assume $K$ to be normalized and let $K_{\max} = \max_{M}K$. Assume that the fixed point component $F_{\max} = K^{-1}(K_{\max})$ is semifree. Then there are classes $a_{B} \in H_{*}(M)$ {and an integer $m_{\max}$} so that
  \begin{equation}
    \label{eq:SeidelToric-byMcDT}
    \m S(\Lambda) = [F_{\max}] \otimes q^{-m_{\max}}t^{K_{\max}} + \sum a_{B} \otimes q^{-m_{\max}-c_{1}(B)}t^{K_{\max}-\omega(B)}
  \end{equation}
    where the sum runs over all spherical classes $B \in \HH^{S}_{2}(M;\Z)$ with $\omega(B)>0$.
\end{theo*}

Recall that semifreeness means that the group acts semifreely: the stabilizer of each point is trivial or the whole circle in a neighbourhood of $F_{\max}$. The integer $m_{\max}$ will not be used below. Let us simply mention that it is determined by the degree of $ \m S(\Lambda)$ and  that it corresponds to the sum of the weights at a point $x \in F_{\max}$. Note that McDuff and Tolman were also able to specify the structure of the lower order terms, under additional assumptions on $\omega$-compatible almost complex structures.

For our purpose, it is enough to notice that, since $[F_{\max}]$ and the $a_{B}$'s are honest classes in $\HH_{*}(M)$, the valuation of $\m S(\Lambda)$ is $K_{\max}$. Moreover, $\Lambda^{-1}$ is generated by $-K$ whose maximal fixed point component is nothing but $(-K)^{-1}(\max_{M}(-K)) = K^{-1}(\min_{M} K) = F_{\min}$. Requiring both extremal fixed point components to be semifree and applying McDuff--Tolman's Theorem to $K$ and $-K$, yield
\begin{equation*}
  \Gamma(\wt\Lambda) = \nu(\m S(\Lambda)) + \nu(\m S(\Lambda^{-1})) = K_{\max} - K_{\min} 
\end{equation*}
where $K_{\min} = \min_{M} K$.
Since $K$ is not constant, $\Gamma(\wt\Lambda) > 0$.
Under the additional assumption that $(M,\omega)$ is rational, Proposition \ref{prop:extension-Seidel} ensures that $\iota_{*}(\wt\Lambda)$ is not trivial in $\pi_{1}\big(\overline\Ham(M,\omega)\big)$. \hfill $\square$

\subsection{Proof of Corollary \ref{theo:CPn}}
\label{sec:proof-theor-CPn}

Recall from \cite{MR1487754} that {there is} an element $h_{n}$ of order $n+1$ in $\pi_{1}\big(\Ham(\CP^{n},\omega_{\mathrm{FS}})\big)$ {which has non-trivial image by the Seidel homomorphism}. We claim that $\Gamma$ is positive on all classes {with non-trivial associated Seidel element}. Since $(\CP^{n},\omega_{\mathrm{FS}})$ is monotone, it is rational and Proposition \ref{prop:extension-Seidel} ensures that $\iota_{*}(h_{n})$ is of order $n+1$ in $\Hambar(\CP^{n},\omega_{\mathrm{FS}})$.

\medskip
We now proceed and show that 
  \begin{equation}\label{eq:Gamma-equals1-for-CPn}
    \forall\,h \in \pi_{1}\big(\Ham(\CP^{n},\omega_{\mathrm{FS}})\big), \quad h \notin \ker(\m S) \; \Longrightarrow \; \Gamma(h) = \Omega
  \end{equation}
  where $\Omega = \omega_{\mathrm{FS}}([\CP^{1}]) > 0$ is the generator of $\langle \omega_{\mathrm{FS}} , \pi_{2}(\CP^{n}) \rangle$.
  This obviously implies our claim, hence concluding the proof.

  We recall that the quantum homology algebra of $\CP^{n}$ is isomorphic to the algebra $\mathrm k[A] / \{A^{n+1}=s^{-1}\}$ with $\mathrm k = \bb k[[s]$, see our  comment on the monotone case at the end of Section \ref{sec:quantum-homology}. Here, $A$ is the hyperplane class of degree $2n-2$ and $s$ is of degree $2N = 2(n+1)$.
  Notice that for any $m\in\Z$, $A^{m}$ has degree $2n-2m$.

  Moreover, recall from Proposition 4.2 of \cite{MR1979584}, that all Seidel elements are monomials of the form $A^{m}s^{\beta}$. Since they belong to $\QH_{2n}(\CP^{n};\Lambda)$, $\beta = \frac{m}{n+1}$ for degree reasons. Notice that when $m$ is a multiple of $n+1$, we get $A^{m}s^{\beta} = [\CP^{n}]$ as expected (see Proposition 4.3 of \cite{MR1979584}).

  We now assume that $m$ is not a multiple of $n+1$. We need to compute $\nu(A^{m}s^{\beta})+\nu(A^{-m}s^{-\beta})$ so that we may assume that $m \geq 0$. Then
  \begin{equation*}
    m = q \cdot (n+1) + r \quad\mbox{ and }\quad -m = -(q+1) \cdot (n+1) + (n+1-r)
  \end{equation*}
  with $q = \lfloor \frac{m}{n+1} \rfloor$, and $r$ and $n+1-r$ are integers in $(0,n+1)$. Hence,
  \begin{equation*}
    A^{m}s^{\frac{m}{n+1}} = (A^{n+1})^{q} \cdot A^{r}s^{\frac{m}{n+1}} = ([\CP^{n}]s^{-1})^{q} \cdot A^{r}s^{\frac{m}{n+1}} =  A^{r}s^{\frac{m}{n+1}-q} 
  \end{equation*}
  whose valuation is $\Omega \cdot (\frac{m}{n+1}-q) = \Omega \cdot (\frac{m}{n+1}-\lfloor \frac{m}{n+1} \rfloor)$ since $A^{r} \in \HH_{*}(\CP^{n};\Q)$.
  Similarly, by replacing $q$ by $-(q+1)$ and $r$ by $n+1-r$, we get that $A^{-m}s^{\frac{-m}{n+1}}$ has valuation $\Omega \cdot (\lfloor \frac{m}{n+1} \rfloor + 1 - \frac{m}{n+1})$, so that $\nu(A^{m}s^{\beta})+\nu(A^{-m}s^{-\beta}) = \Omega$. 
  
  Hence, $\Gamma(h)=\Omega$ whenever $h$ has a non-trivial associated Seidel element which proves (\ref{eq:Gamma-equals1-for-CPn}) and ends the proof of Corollary \ref{theo:CPn}. \hfill $\square$

\subsection{About Example \ref{example:non-monotone-S2xS2} and the proof of Theorem \ref{theo:Hirzebruch-surfaces}}\label{sec:computations-Hirz-Surfaces}

In this section, we compute the Seidel elements associated to all elements of the fundamental group of the group of Hamiltonian diffeomorphisms of all Hirzebruch surfaces. This completes partial computations from \cite{MR1959582}, \cite{ostrover}, and \cite{MR3681107}.

  Recall that in \cite{MR3798327}, Anjos and the third author computed Seidel elements of many 4-dimensional toric manifolds starting from the fundamental aforementioned result from McDuff and Tolman. In \cite{MR3681107}, these computations were used to prove that Seidel's morphism detects all the generators of the fundamental group of the Hamiltonian diffeomorphism group of all Hirzebruch surfaces, namely all symplectic products of $S^{2}\times S^{2}$ and 1-point blow-ups of $\CP^{2}$.

 We now show that, when we consider $\Gamma$ rather than the Seidel elements themselves, the situation is more subtle:
 \begin{itemize}
 \item $\Gamma$ vanishes on certain essential loops of Hamiltonian diffeomorphisms of all products of spheres except the monotone one, this will justify Example \ref{example:non-monotone-S2xS2} and Fact (F1) of Proposition \ref{prop:facts-on-spectral-norm} from the introduction ;
 \item $\Gamma$ is positive on all essential loops of Hamiltonian diffeomorphisms of the 1-point blow-ups of $\CP^{2}$, this will prove Theorem \ref{theo:Hirzebruch-surfaces} and justify Fact (F2) of Proposition \ref{prop:facts-on-spectral-norm}.
\end{itemize}
The computations done in the latter case will also clarify Facts (F3) to (F6) of Proposition \ref{prop:facts-on-spectral-norm}.
\begin{remark}
Note that all the quantities $\m S$, $\nu$, $\Gamma$, $\wt\gamma$ depend on the symplectic forms, themselves being parameterized by a real number $\mu$ in both cases. However, we omit $\mu$ in the notation for brevity throughout this section. 
\end{remark}

\paragraph{Even Hirzebruch surfaces.}
Recall that $\F_{2k}^{\mu}$ is identified to $M=S^{2}\times S^{2}$ endowed with the symplectic form $\omega_{\mu}$ with area $1$ on the first factor and $\mu$ on the second.

We denote by $u = [S^{2}\times \{\mathrm{pt}\}] \otimes q$ and $v = [\{\mathrm{pt}\}\times S^{2}] \otimes q$ the degree 4 quantum homology classes induced by each component of the product $M=S^{2}\times S^{2}$. The quantum homology algebra of $(M,\omega_{\mu})$ is
\begin{equation*}
  \QH_{*}(M,\omega_{\mu}) \simeq \Lambda^{\mathrm{univ}}[u,v] / \langle u^{2} = t^{-1}, v^{2}= t^{-\mu} \rangle \,.
\end{equation*}

The fundamental group of the group of Hamiltonian diffeomorphisms of $\F_{2k}^{\mu}$ was computed in \cite{MR1775741}.

\smallskip
\noindent $\bullet$ When $\mu = 1$, the manifold $(M,\omega_{1})$ is monotone hence rational, and the fundamental group of $\Ham(M,\omega_{1})$ is generated by two elements of order 2, each generated by one factor of the product. These generators are Hamiltonian circle actions which satisfy the assumption of Theorem~\ref{theo:general-circle-action} so that they induce non-trivial elements in $\Hambar(M,\omega_{1})$. Since they are of order 2, there is nothing more to prove. 

\smallskip
\noindent $\bullet$ When $\mu > 1$ and rational, $(M,\omega_{\mu})$ is rational and $\pi_{1}\big(\Ham(M,\omega_{\mu})\big)$ is generated by the same order-2 circle actions, together with a third Hamiltonian circle action of infinite order $\Lambda$ (denoted $\Lambda^{2}_{e_{1}}$ in \cite{MR3681107}). The Seidel element of the latter and of its inverse are given by
\begin{equation*}
  \m S(\Lambda) = (u+v) \otimes t^{\frac 12 - \epsilon} \quad \mbox{ and } \quad \m S(\Lambda)^{-1} = (u-v) \otimes \frac{t^{\frac 12 + \epsilon}}{1-t^{1-\mu}} 
\end{equation*}
with $\epsilon = \frac 1{6\mu}$. Hence, for any positive integer $\ell$,
\begin{equation*}
  \m S(\Lambda)^{\ell} = \sum_{k=0}^{\ell} \Big( \!\!
  \begin{array}{c}     \ell \\ k   \end{array} \!\!
  \Big) u^{k} v^{\ell-k} t^{(\frac 12-\epsilon)\ell} \,.
\end{equation*}
Notice that, depending on the parity of $k$ and $\ell$ and  up to some power of $t$, $u^{k}v^{\ell-k}$ is $[M]$, $u$, $v$, or $uv$, so that no products of powers of $u$ and $v$ vanish. Moreover, for all $k$ and $\ell$,
$$\nu \big( u^{k}v^{\ell-k}t^{(\frac 12-\epsilon)\ell} \big)= -\Big\lfloor \frac k2 \Big\rfloor - \mu \Big\lfloor \frac{\ell-k}2 \Big\rfloor - \Big(\frac 12-\epsilon\Big)\ell \,.$$
Since $\mu > 1$, we get that $\nu(\m S(\Lambda)^{\ell}) = - \lfloor \frac \ell 2 \rfloor + (\frac 12-\epsilon)\ell$.

The computation of $\nu(\m S(\Lambda)^{-\ell})$ is similar. Only notice additionally that $\frac 1{1-t^{1-\mu}} = 1 + t^{1-\mu} + t^{2(1-\mu)} + \ldots$ so that, since $1-\mu <0$, $\nu\big((\frac 1{1-t^{1-\mu}})^{\ell}\big) = 0$. Hence, we get $\nu(\m S(\Lambda)^{-\ell}) = - \lfloor \frac \ell 2 \rfloor + (\frac 12+\epsilon)\ell$ which yields
\begin{equation*}
  \Gamma(\Lambda^{\ell}) = \ell - 2 \Big\lfloor \frac\ell 2\Big\rfloor = \left \{
    \begin{array}{l}
      0 \mbox{ if $\ell$ is even},\\
      1 \mbox{ if $\ell$ is odd}.
    \end{array}
    \right.
\end{equation*}
Hence, we see that $\Gamma$ detects only odd powers of $\Lambda$, it follows that $\iota_{*}([\Lambda]) \neq 0$ (but it could be of order 2). \hfill $\square$

\paragraph{Odd Hirzebruch surfaces.} We follow the conventions from \cite[Section 3.2]{MR3681107}: $\bb F_{2k+1}^{\mu}$
  is identified with $\CP^{2}\#\overline{\CP^{2}}$ endowed with the symplectic form $\omega'_{\mu}$ such that
\begin{itemize}[itemsep=.5pt,topsep=4pt]
\item the exceptional divisor $B$ of self-intersection $-1$ has area $\mu >0$,
\item the fiber $F$ has area $1$,
\item the projective line $B+F$ has area $\mu+1$.
\end{itemize}
As a vector space, its quantum homology is generated by:
$u_{0}=[\mathrm{pt}] \otimes q^{2}$, $u=F\otimes q$, $u_{3}=B\otimes q$, and $\bb 1$. It is also convenient to denote by $u_{1} = (B+F)\otimes q = u+u_{3}$. These classes satisfy the following relations\footnote{The last one does not appear as such in \cite{MR3681107} as it was not needed to get the expression of the quantum homology algebra. It is necessary here and can easily be computed with the methods used there. See also \cite[Remark 5.6]{MR2210662} and \cite[Section 3.3]{ostrover} in which it is explicitly computed \emph{under different conventions} though.}:
\begin{equation}\label{eq:relations-generators-QH-odd-Hirz}
  u_{1}\, u_{3} = t^{-1} \,, \quad u^{2} = u_{3}\,t^{-\mu} \,, \quad u^{-1} = u_{0}\,t^{\mu+1} \,.
\end{equation}
As an algebra, the quantum homology is given as the quotient
\begin{equation*}
 \QH_{*} (\bb F^{\mu}_{2k+1}) \simeq \Lambda^{\mathrm{univ}} [u] / \langle u^{4}t^{2\mu} + u^{3}t^{\mu} - t^{-1} \rangle \,.
\end{equation*}
 Last (but not least!), recall from \cite{MR1775741} that the fundamental group of the group of Hamiltonian diffeomorphisms of $\bb F_{2k+1}^{\mu}$ is generated by a single class of infinite order. This class is induced by a circle action $\Lambda$ whose associated Seidel element is $\m S(\Lambda) = u^{-1}t^{-\varepsilon}$ where $\varepsilon = \frac{3\mu^{2}+3\mu+1}{3(1+2\mu)}$. Hence we get, for all integers $p$,
\begin{equation*}
  \Gamma(\Lambda^{p}) = \nu (\m S (\Lambda^{p})) + \nu (\m S (\Lambda^{-p})) = \nu(u^{-p}) + \nu(u^{p}) \,.
\end{equation*}
The following proposition collects the results of the computations of $\Gamma(\Lambda^{p})$ for all $p$ (greater than some $p_{0}$) for all possible values of the parameter $\mu > 0$. 

\begin{prop}\label{prop:formulas-for-Gamma-1ptBUofCP2}
  
\begin{itemize}[leftmargin=12pt]
\item For $0 < \mu \leq \frac 12$,
  \begin{align}
    \label{eq:Gamma_for_mu_leq1half}
    \forall p \geq 7,\quad\Gamma(\Lambda^{p}) = -2 \Big( \Big\lfloor \frac{p-1}{3} \Big\rfloor -1\Big) \cdot \mu + \Big( \Big\lfloor \frac{p-1}{3} \Big\rfloor +1\Big) \,.
  \end{align}
\item For $\mu > \frac 12$, the value of $\Gamma(\Lambda^{p})$ depends on the rest $\ell$ of the Euclidean division of $p$ by $4$, and on the sign of $\mu-1$. Namely, for all $p \geq 12$, $\Gamma(\Lambda^{p})$ is given by the following table:
  \begin{align}
        \begin{array}[t]{c||c|c} 
      \phantom{\Big\lfloor} \ell \phantom{\Big\lfloor} & \frac 12 < \mu \leq 1 & 1 < \mu \\
      \hline
      0 & 2 & 2 \\
      1 & 2 & \mu + 1 \\
      2 & -2\mu+3 & 1 \\
      3 & 2 & \mu+1
    \end{array}
  \end{align}
\end{itemize}
\end{prop}

This proposition shows that, for any $\mu >0$, $\Gamma(\Lambda^{p}) \neq 0$ for $p$ big enough\footnote{Computing the first few terms shows that this fact actually holds for all integers $p$. The formulas we present below, however, do not hold for small integers $p$.} which yields that $\iota_{*}([\Lambda])$ is of infinite order in $\pi_{1}\big(\overline\Ham(\bb F_{2k+1}^{\mu})\big)$ for any positive, \textit{rational} number $\mu$.  Notice that Theorem \ref{theo:Hirzebruch-surfaces} and the facts (F2) to (F6) from Proposition \ref{prop:facts-on-spectral-norm} follow from this proposition.

The graph below illustrates Proposition \ref{prop:formulas-for-Gamma-1ptBUofCP2}.
\begin{center}
\footnotesize
\def\svgwidth{120mm}
\begingroup%
  \makeatletter%
  \providecommand\color[2][]{%
    \errmessage{(Inkscape) Color is used for the text in Inkscape, but the package 'color.sty' is not loaded}%
    \renewcommand\color[2][]{}%
  }%
  \providecommand\transparent[1]{%
    \errmessage{(Inkscape) Transparency is used (non-zero) for the text in Inkscape, but the package 'transparent.sty' is not loaded}%
    \renewcommand\transparent[1]{}%
  }%
  \providecommand\rotatebox[2]{#2}%
  \newcommand*\fsize{\dimexpr\f@size pt\relax}%
  \newcommand*\lineheight[1]{\fontsize{\fsize}{#1\fsize}\selectfont}%
  \ifx\svgwidth\undefined%
    \setlength{\unitlength}{340.15748031bp}%
    \ifx\svgscale\undefined%
      \relax%
    \else%
      \setlength{\unitlength}{\unitlength * \real{\svgscale}}%
    \fi%
  \else%
    \setlength{\unitlength}{\svgwidth}%
  \fi%
  \global\let\svgwidth\undefined%
  \global\let\svgscale\undefined%
  \makeatother%
  \begin{picture}(1,0.4)%
    \lineheight{1}%
    \setlength\tabcolsep{0pt}%
    \put(0,0){\includegraphics[width=\unitlength,page=1]{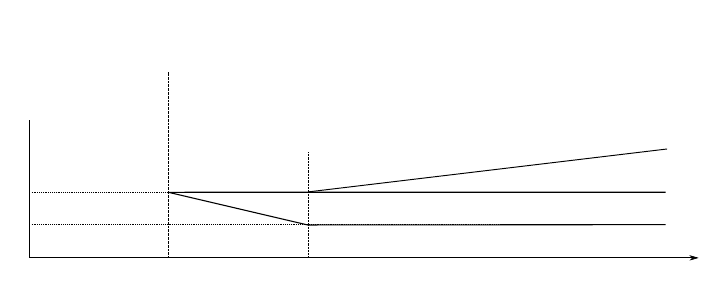}}%
    \put(0.22446793,0.00441971){\color[rgb]{0,0,0}\makebox(0,0)[lt]{\lineheight{1.25}\smash{\begin{tabular}[t]{l}$\frac 12$\end{tabular}}}}%
    \put(0.42523514,0.00882943){\color[rgb]{0,0,0}\makebox(0,0)[lt]{\lineheight{1.25}\smash{\begin{tabular}[t]{l}$1$\end{tabular}}}}%
    \put(0.9413072,0.01323916){\color[rgb]{0,0,0}\makebox(0,0)[lt]{\lineheight{1.25}\smash{\begin{tabular}[t]{l}$\mu$\end{tabular}}}}%
    \put(0.00726404,0.07887325){\color[rgb]{0,0,0}\makebox(0,0)[lt]{\lineheight{1.25}\smash{\begin{tabular}[t]{l}$1$\end{tabular}}}}%
    \put(0.09738865,0.05735979){\color[rgb]{0,0,0.99607843}\makebox(0,0)[lt]{\lineheight{1.25}\smash{\begin{tabular}[t]{l}$p=7, 8, 9$\end{tabular}}}}%
    \put(0.11164307,0.09322802){\color[rgb]{0,0,0.99607843}\makebox(0,0)[lt]{\lineheight{1.25}\smash{\begin{tabular}[t]{l}$p=10, 11, 12$\end{tabular}}}}%
    \put(0.1933373,0.34884342){\color[rgb]{0,0,0.99607843}\makebox(0,0)[lt]{\lineheight{1.25}\smash{\begin{tabular}[t]{l}$p=3(n-1)+k$, \\     $k=1, 2, 3$\end{tabular}}}}%
    \put(0.37579006,0.20388489){\color[rgb]{0,0,0.99607843}\makebox(0,0)[lt]{\lineheight{1.25}\smash{\begin{tabular}[t]{l}$p \equiv 2 \!\mod 4$\end{tabular}}}}%
    \put(0.55967102,0.17844919){\color[rgb]{0,0,0.99607843}\makebox(0,0)[lt]{\lineheight{1.25}\smash{\begin{tabular}[t]{l}$p \equiv 1 \!\mod 2$\end{tabular}}}}%
    \put(0.85191023,0.18702461){\color[rgb]{0,0,0}\rotatebox{8.9019156}{\makebox(0,0)[lt]{\lineheight{1.25}\smash{\begin{tabular}[t]{l}$\mu+1$\end{tabular}}}}}%
    \put(0.30171227,0.09157752){\color[rgb]{0,0,0}\rotatebox{-14.331958}{\makebox(0,0)[lt]{\lineheight{1.25}\smash{\begin{tabular}[t]{l}$-2\mu+3$\end{tabular}}}}}%
    \put(0.00726404,0.12515173){\color[rgb]{0,0,0}\makebox(0,0)[lt]{\lineheight{1.25}\smash{\begin{tabular}[t]{l}$2$\end{tabular}}}}%
    \put(0.00726404,0.17015683){\color[rgb]{0,0,0}\makebox(0,0)[lt]{\lineheight{1.25}\smash{\begin{tabular}[t]{l}$3$\end{tabular}}}}%
    \put(0.00726404,0.21697405){\color[rgb]{0,0,0}\makebox(0,0)[lt]{\lineheight{1.25}\smash{\begin{tabular}[t]{l}$4$\end{tabular}}}}%
    \put(0.00726404,0.35744672){\color[rgb]{0,0,0}\makebox(0,0)[lt]{\lineheight{1.25}\smash{\begin{tabular}[t]{l}$n$\end{tabular}}}}%
    \put(0,0){\includegraphics[width=\unitlength,page=2]{dessin.pdf}}%
    \put(0.55655797,0.35543031){\color[rgb]{0,0,0}\makebox(0,0)[lt]{\lineheight{1.25}\smash{\begin{tabular}[t]{l}The values of $\wt\gamma([\Lambda^{p}])$ as functions\\of the symplectic form $\omega'_\mu$.\end{tabular}}}}%
  \end{picture}%
\endgroup%

\end{center}

\begin{remark} 
    It is interesting to compare the first item in Proposition \ref{prop:formulas-for-Gamma-1ptBUofCP2} with the results of Ostrover \cite{ostrover}. Indeed, he establishes that for $\mu > \frac 12$, the quantum homology $QH_{4}(\F^\mu_{2k+1})$ is a field, while it \emph{splits} into field summands for $\mu < \frac 12$. 
    Moreover, he proves that the restriction of $\Gamma$ to any field summand is bounded. Our result shows that $\Gamma$ is however not bounded on the whole quantum homology in the case $\mu<\frac12$.
\end{remark}

\medskip
It remains to prove Proposition \ref{prop:formulas-for-Gamma-1ptBUofCP2} which follows from intermediate computations whose results are collected below, depending on the value of $\mu$.
\begin{align*}
  \begin{array}{c|c|c}
    & \nu(u^{-p}) & \nu(u^{p}) \\
    \hline
  \phantom{\Big(}\!\!\! 0 < \mu \leq \tfrac 12 & \big(p-2 \big\lfloor \frac{p-1}3 \big\rfloor \big) \mu + \big( \big\lfloor \frac{p-1}3 \big\rfloor +1 \big) & -(p-2) \mu \\
  \hline
  \phantom{\Big(}\!\!\! \tfrac 12 < \mu \leq 1 & \big(p-2 \big\lfloor \frac{p+2}4 \big\rfloor \big) \mu + \big( \big\lfloor \frac{p+2}4 \big\rfloor +1 \big) & -\big(p-2 \big\lfloor \frac{p+1}4 \big\rfloor \big) \mu - \big( \big\lfloor \frac{p+1}4 \big\rfloor -1 \big) \\
  \hline
  \phantom{\Big(}\!\!\! 1 < \mu \phantom{\leq3} & \big\lfloor \frac{p+1}{2} \big\rfloor  \mu + \big( \big\lfloor \frac p4 \big\rfloor +1\big) & -\big\lfloor \frac{p}{2} \big\rfloor  \mu - \big\lfloor \frac {p-1}4 \big\rfloor
\end{array}
\end{align*}
\noindent These formulas are proved by induction. Let us first focus on the latter two cases.

\bigskip
\noindent \fbox{We assume $\mu > \tfrac 12$.}

\medskip
We wish to show that for $p$ big enough, $\nu(u^{-p})$ and $\nu(u^{p})$ are of the form specified in the table above, depending on the sign of $\mu-1$.
The values of $\nu(u^{-p})$ will be extracted from the proof of the following lemma.
\begin{lemma}\label{lem:higherorderterm_negativemultipleof4}
  For all integers $q \geq 3$, there are polynomials $P_{i}^{q}$ in $t$, so that
  \begin{equation*}
    u^{-4q} = u_{0} \,P_{0}^{q}(t) + u_{1} \,P_{1}^{q}(t) + u \,P_{2}^{q}(t) + \mathbb 1 P_{3}^{q}(t) .
  \end{equation*}
These polynomials are of the form
\begin{align*}
  P_{0}^{q}(t) &= \alpha_{0}^{q} \,t^{2q\mu+(q+1)} &+\; \mbox{lower order terms} \\
  P_{1}^{q}(t) &= \alpha_{1}^{q} \,t^{(2q-1)\mu+(q+1)} &+\; \mbox{lower order terms} \\
  P_{2}^{q}(t) &= \alpha_{2}^{q} \,t^{(2q-1)\mu+(q+1)} &+\; \mbox{lower order terms} \\
  P_{3}^{q}(t) &= \alpha_{3}^{q} \,t^{2q\mu+q} &+\; \mbox{lower order terms} 
\end{align*}
where all coefficients $\alpha_{i}^{q}$ are positive.
\end{lemma}

Its proof consists of tedious but straigthforward computations. First, notice that the relations (\ref{eq:relations-generators-QH-odd-Hirz}), give $u^{-1} = u_{0}\, t^{\mu+1}$ and the following multiplication table
\begin{equation}\label{eq:multiplication-table-u-1}
  \begin{array}{c||c|c|c|c}
   a & u_{0} & u_{1} & u & \mathbb 1 \\ \hline
    \phantom{\Big(}\!\! a \ast u^{-1} & u_{1} & u\,t^{\mu} + \mathbb 1 & \mathbb 1 & u_{0}\,t^{\mu+1} 
  \end{array}
\end{equation}
Next, compute explicitly $u^{-4q}$ for $q=3$ as {base case for a proof by induction on $q$:}
\begin{align*}
  u^{-12} &= 3u_{0}\, t^{6\mu+4} + 3u_{1}\, t^{5\mu+4} + u\, t^{5\mu+4} + \mathbb 1 (t^{6\mu+3}+t^{4\mu+4}) \,
\end{align*}
which is of the form specified in Lemma \ref{lem:higherorderterm_negativemultipleof4} since $\mu \geq \frac 12$ (otherwise $P_{3}^{3}$ would have valuation $4\mu+4$).

Finally, assume $u^{-4q}$ has the {form} given in the lemma and compute the 4 next powers of $u$. Up to lower order terms, we get (to ease the reading we denote $P^{q}_{i}(t)$ simply by $P_{i}$):
\begin{align}\label{eq:successive-computations-negative-powers}
    u^{-4q-1} &= u_{0} P_{3} t^{\mu+1}  + u_{1} P_{0} + u  P_{1} t^{\mu}  + \mathbb 1 (P_{1} + P_{2}) \\ \nonumber
    u^{-4q-2} &= u_{0} (P_{1} + P_{2}) t^{\mu+1}   + u_{1} P_{3} t^{\mu+1} + u P_{0} t^{\mu}  + \mathbb 1 (P_{0}+P_{1}t^{\mu}) \\ \nonumber
    u^{-4q-3} &= u_{0}(P_{0}+P_{1}t^{\mu})t^{\mu+1}  + u_{1} (P_{1}+P_{2})t^{\mu+1}  + u P_{3}t^{2\mu+1}  + \mathbb 1 (P_{0} + P_{3}t)t^{\mu}
\end{align}
This yields for $u^{-4(q+1)}$:
\begin{align*}
  P_{0}^{q+1}(t) &= P_{3}^{q}(t)\cdot t^{2\mu+2} + P_{0}^{q}(t)\cdot t^{2\mu+1} & + \mbox{ l.o.t.}\\
  P_{1}^{q+1}(t) &= P_{0}^{q}(t)\cdot t^{\mu+1} + P_{1}^{q}(t)\cdot t^{2\mu+1} & + \mbox{ l.o.t.}\\
  P_{2}^{q+1}(t) &= (P_{1}^{q}(t)+P_{2}^{q}(t)) \cdot t^{2\mu+1} & + \mbox{ l.o.t.}\\
  P_{3}^{q+1}(t) &= P_{3}^{q}(t) \cdot t^{2\mu+1} + (P_{1}^{q}(t)+P_{2}^{q}(t)) \cdot t^{\mu+1}  & + \mbox{ l.o.t.}
\end{align*}
whose respective leading order terms are
\begin{align*}
  P_{0}^{q+1} : \quad &(\alpha_{0}^{q}+\alpha_{3}^{q}) \,t^{2(q+1)\mu+(q+1)+1}\\
  P_{1}^{q+1}: \quad &(\alpha_{0}^{q}+\alpha_{1}^{q}) \, t^{(2(q+1)-1)\mu+((q+1)+1)}\\
  P_{2}^{q+1}: \quad & (\alpha_{1}^{q}+\alpha_{2}^{q})\, t^{(2(q+1)-1)\mu+((q+1)+1)}\\
  P_{3}^{q+1}: \quad &\alpha_{3}^{q}\, t^{2(q+1)\mu+(q+1)}
\end{align*}
since all the coefficients $\alpha_{i}^{q}$ are positive and since $\mu > \tfrac 12$ so that $(P_{1}^{q}(t)+P_{2}^{q}(t)) \cdot t^{\mu+1}$ only consists of lower order terms of $P_{3}^{q+1}$. 

\medskip
This proves the lemma. The formula giving $\nu(u^{-p})$ now follows from the lemma together with the valuation of intermediate powers of $u^{-1}$ which easily follow from the computations (\ref{eq:successive-computations-negative-powers}). Indeed, depending on the sign of $\mu - 1$, we get:
\begin{align}\label{eq:intermediate-computations-negative-powers}
  \begin{array}{c|c|c}
    & \tfrac 12 < \mu \leq 1 & 1 < \mu  \phantom{\Big(} \\
    \hline
  \phantom{\Big(} p = 4q\phantom{+1} & \multicolumn{2}{c}{2q\mu + (q+1)} \\
    \hline
  \phantom{\Big(} p = 4q+1 & \multicolumn{2}{c}{(2q+1) \mu + (q+1)}  \\
  \hline
  \phantom{\Big(} p=4q+2 & 2q\mu + (q+2) & (2q+1) \mu + (q+1) \\
  \hline
  \phantom{\Big(} p=4q+3 & (2q+1)\mu + (q+2) & 2(q+1) \mu + (q+1)
\end{array}
\end{align}

\medskip
We now turn to the valuation of positive powers of $u$. Symmetrically to the negative powers, $\nu(u^{p})$ will be extracted from the proof of the following lemma.
\begin{lemma}\label{lem:higherorderterm_positivemultipleof4}
  For all integers $q \geq 2$, there are polynomials $Q_{i}^{q}$ in $t$, so that
  \begin{equation*}
    u^{4q} = u_{0} \,Q_{0}^{q}(t) + u \,Q_{2}^{q}(t) + u_{3} \,Q_{3}^{q}(t) + \mathbb 1 Q_{4}^{q}(t) .
  \end{equation*}
These polynomials are of the form
\begin{align*}
  Q_{0}^{q}(t) &= -\beta_{0}^{q} \,t^{-2q\mu-(q-1)} &+\; \mbox{lower order terms} \\
  Q_{2}^{q}(t) &= -\beta_{2}^{q} \,t^{-(2q+1)\mu-(q-1)} &+\; \mbox{lower order terms} \\
  Q_{3}^{q}(t) &= \beta_{3}^{q} \,t^{-(2q+1)\mu-(q-1)} &+\; \mbox{lower order terms} \\
  Q_{4}^{q}(t) &= \beta_{4}^{q} \,t^{-2q\mu-q} &+\; \mbox{lower order terms} 
\end{align*}
where all coefficients $\beta_{i}^{q}$ are positive.
\end{lemma}
Notice that we changed the basis of the cohomology, this yields easier computations and a slight notational discrepancy.
The relevant multiplication table now is
\begin{equation}\label{eq:multiplication-table-u}
  \begin{array}{c||c|c|c|c}
   a & u_{0} & u & u_{3} & \mathbb 1 \\ \hline
    \phantom{\Big(}\!\! a \ast u & \bb 1\,t^{-\mu-1} & u_{3}\,t^{-\mu} & u_{0}-u_{3}\,t^{-\mu} & u 
  \end{array}
\end{equation}
We initialize the induction with $u^{4q}$ for $q=2$:
\begin{align*}
  u^{8} = -u_{0}\, (2t^{-4\mu-1}+t^{-6\mu}) - u\, t^{-5\mu-1} + u_{3}\, (3t^{-5\mu-1}+t^{-7\mu}) + \mathbb 1 (t^{-4\mu-2}+t^{-6\mu-1})
\end{align*}
which is seen to be of the form specified by Lemma \ref{lem:higherorderterm_positivemultipleof4}, again, because $\mu \geq \tfrac 12$.

We now assume $u^{4q}$ has the expression given above and compute the 4 next powers of $u$. To ease the reading we denote $Q^{q}_{i}(t)$ simply by $Q_{i}$ and we introduce the notation $Q_{2,3}$ for $Q_{2}-Q_{3}$. We get, up to lower order terms:
\begin{align}\label{eq:successive-computations-positive-powers}
    u^{4q+1} &= u_{0} Q_{3} + u Q_{4} + u_{3} Q_{2,3} t^{-\mu}  + \mathbb 1 Q_{0} t^{-\mu-1} \\ \nonumber
    u^{4q+2} &= u_{0}  Q_{2,3} t^{-\mu}  + u Q_{0} t^{-\mu-1} + u_{3} (Q_{4}t^{-\mu} -  Q_{2,3}t^{-2\mu})  + \mathbb 1 Q_{3}t^{-\mu-1} \\ \nonumber
  u^{4q+3} &= u_{0}(Q_{4}t^{-\mu} -  Q_{2,3}t^{-2\mu}) + u Q_{3}t^{-\mu-1}  \\ \nonumber
             &\hspace{20mm}+ u_{3}(Q_{0}t^{-2\mu-1} - Q_{4} t^{-2\mu} + Q_{2,3}t^{-3\mu}) + \mathbb 1 Q_{2,3}t^{-2\mu-1}
\end{align}
This yields for $u^{4(q+1)}$ (again, \textit{up to lower order terms}): 
\begin{align*}
  Q_{0}^{q+1}(t) &= Q_{0}^{q}(t)\cdot t^{-2\mu-1} - Q_{4}^{q}(t)\cdot t^{-2\mu}  +(Q_2^q(t)-Q_3^q(t))\cdot t^{-3\mu}\\
  Q_{2}^{q+1}(t) &= (Q_{2}^{q}(t)-Q_{3}^{q}(t))\cdot t^{-2\mu-1}\\
  Q_{3}^{q+1}(t) &= -Q_{0}^{q}(t)\cdot t^{-3\mu-1}+Q_{3}^{q}(t) \cdot t^{-2\mu-1}+Q_{4}^{q}(t)\cdot t^{-3\mu} -(Q_2^q(t)-Q_3^q(t))\cdot t^{-4\mu}\\
  Q_{4}^{q+1}(t) &= Q_{4}^{q}(t) \cdot t^{-2\mu-1} -(Q_2^q(t)-Q_3^q(t))\cdot t^{-3\mu-1}
\end{align*}
whose respective leading order terms are
\begin{align*}
  Q_{0}^{q+1} : \quad &-(\beta_{0}^{q}+\beta_{4}^{q}) \,t^{-2(q+1)\mu-((q+1)-1)}\\
  Q_{2}^{q+1}: \quad &-(\beta_{2}^{q}+\beta_{3}^{q}) \, t^{-(2(q+1)+1)\mu-((q+1)-1)}\\
  Q_{3}^{q+1}: \quad & (\beta_{0}^{q}+\beta_{3}^{q}+\beta_{4}^{q})\, t^{-(2(q+1)+1)\mu-((q+1)-1)}\\
  Q_{4}^{q+1}: \quad &\beta_{4}^{q}\, t^{-2(q+1)\mu-(q+1)}
\end{align*}
since all the coefficients $\beta_{i}^{q}$ are positive and since $\mu > \tfrac 12$.

\medskip
This proves Lemma \ref{lem:higherorderterm_positivemultipleof4}. The expression of $\nu(u^{p})$ then follows from it, together with the valuation of intermediate powers of $u$ which easily follow from the computations (\ref{eq:successive-computations-positive-powers}). Namely, we get
\begin{align}\label{eq:intermediate-computations-positive-powers}
  \begin{array}{c|c|c}
    & \tfrac 12 < \mu \leq 1 & 1 < \mu  \phantom{\Big(} \\
    \hline
  \phantom{\Big(} p = 4q\phantom{+1} & \multicolumn{2}{c}{-2q\mu - (q-1)} \\
  \hline
  \phantom{\Big(} p=4q+1 & -(2q+1) \mu - (q-1)  & -2q\mu-q \\
  \hline
  \phantom{\Big(} p=4q+2 & -2(q+1) \mu - (q-1) & -(2q+1)\mu - q \\
    \hline
  \phantom{\Big(} p = 4q+3 & \multicolumn{2}{c}{-(2q+1) \mu - q}  
\end{array}
\end{align}

Together with the expression of $\nu(u^{-p})$ given by Table (\ref{eq:intermediate-computations-negative-powers}), this concludes the proof of Proposition \ref{prop:formulas-for-Gamma-1ptBUofCP2} in the case $\mu > \tfrac 12$.

\bigskip
\noindent\fbox{We assume $0 < \mu \leq \tfrac 12$.}
\medskip

The sketch of the proof in this case is similar, the main lemmas only need a few adjustments.

\medskip
For negative powers of $u$, the situation is perfectly similar to the previous cases, except that the results depend on the remainder of $p$ mod $3$ rather than $4$.
\begin{lemma}\label{lem:higherorderterm_negativemultipleof3_mulessthan1half}
  For all integers $q \geq 3$, there are polynomials $P_{i}^{q}$ in $t$, so that
  \begin{equation*}
    u^{-3q-1} = u_{0} \,P_{0}^{q}(t) + u_{1} \,P_{1}^{q}(t) + u \,P_{2}^{q}(t) + \mathbb 1 P_{3}^{q}(t) .
  \end{equation*}
These polynomials are of the form
\begin{align*}
  P_{0}^{q}(t) &= t^{(q+1)\mu+(q+1)} &+\; \mbox{lower order terms}\phantom{\,.} \\
  P_{1}^{q}(t) &= \tfrac{(q-1)(q-2)}2 \,t^{(q+2)\mu+q} &+\; \mbox{lower order terms}\phantom{\,.} \\
  P_{2}^{q}(t) &= (q-1) \,t^{(q+2)\mu+q} &+\; \mbox{lower order terms} \phantom{\,.} \\
  P_{3}^{q}(t) &= q \,t^{(q+1)\mu+q} &+\; \mbox{lower order terms}\,.
\end{align*}
\end{lemma}

We use the multiplication table (\ref{eq:multiplication-table-u-1}) to compute
\begin{equation*}
  u^{-10} = u_{0}\, t^{4\mu+4} + u_{1}\, t^{5\mu+3} + 2u \, t^{5\mu+3} + 3\bb 1\,t^{4\mu+3} \,
\end{equation*}
which is indeed of the expected form.

The expressions of the following 3 powers $u^{-3q-2}$, $u^{-3q-3}$, and $u^{-3(q+1)-1}$ are given from $u^{-3q-1}$ by equation (\ref{eq:successive-computations-negative-powers}). From this we conclude that
\begin{align*}
  P_{0}^{q+1} : \quad & t^{(q+2)\mu+(q+2)} & + \mbox{ l.o.t.}\\ 
  P_{1}^{q+1}: \quad & \big( \tfrac{(q-1)(q-2)}2 +(q-1) \big) \,t^{(q+3)\mu+(q+1)}   & + \mbox{ l.o.t.}\\
  P_{2}^{q+1}: \quad & q \,t^{(q+3)\mu+(q+1)}  & + \mbox{ l.o.t.}\\
  P_{3}^{q+1}: \quad & (q+1)\,t^{(q+1)\mu+(q+1)}  & + \mbox{ l.o.t.}
\end{align*}
since $\mu \leq \frac 12$. Notice that this yields the expected expression of $P_{1}^{q+1}$ since $\tfrac{(q-1)(q-2)}2$ is the sum of all integers between $1$ and $q-2$.

This concludes the proof of the lemma, from which we immediately get that $\nu(u^{-3q-1}) = (q+1)\mu+(q+1)$. The other necessary valuations can be extracted from (\ref{eq:successive-computations-negative-powers}):
\begin{align}\label{eq:neq-powers-mu-leq1half}
  \nu(u^{-3q-2}) = (q+2)\mu+(q+1) \;\mbox{ and }\; \nu(u^{-3q-3}) = (q+3)\mu+(q+1) \,.
\end{align}

\medskip
For positive powers of $u$, the situation is somehow easier since we can directly compute the valuation of each power of $u$. The relevant lemma is the following.
\begin{lemma}\label{lem:higherorderterm_positivemultipleof3_mulessthan1half}
  For all integers $p \geq 5$, there are polynomials $Q_{i}^{p}$ in $t$, so that
  \begin{equation*}
    u^{p} = u_{0} \,Q_{0}^{p}(t) + u \,Q_{2}^{p}(t) + u_{3} \,Q_{3}^{p}(t) + \mathbb 1 Q_{4}^{p}(t) .
  \end{equation*}
These polynomials are of the form
\begin{align*}
  Q_{0}^{p}(t) &= (-1)^{p+1}\,t^{-(p-2)\mu} &+\; \mbox{lower order terms} \phantom{\,.}\\
  Q_{2}^{p}(t) &= (-1)^{p+1}\,t^{-(p-3)\mu-1} &+\; \mbox{lower order terms}\phantom{\,.} \\
  Q_{3}^{p}(t) &= (-1)^{p} \,t^{-(p-1)\mu} &+\; \mbox{lower order terms}\phantom{\,.} \\
  Q_{4}^{p}(t) &= (-1)^{p} \,t^{-(p-2)\mu-1} &+\; \mbox{lower order terms} \,. 
\end{align*}
\end{lemma}

Using the multiplication table (\ref{eq:multiplication-table-u}), we compute
\begin{equation*}
  u^{5} = u_{0}t^{-3\mu} + ut^{-2\mu-1} - u_{3}t^{-4\mu} - \bb 1 t^{-3\mu-1} 
\end{equation*}
and we compute $u^{p+1}$ from $u^{p}$
\begin{equation*}
  u^{p+1} = u_{0}\,Q_{3}^{p}(t) + u\, Q_{4}^{p}(t) + u_{3\,}(Q_{2}^{p}(t)-Q_{3}^{p}(t))\, t^{-\mu} + \bb 1 \, Q_{0}^{p}(t)\, t^{-\mu-1}  + \mbox{ l.o.t.}
\end{equation*}
Notice that $Q_{3}^{p+1}$ is of the expected form because $\mu \leq \frac 12$.

\medskip
This proves the lemma and, together with (\ref{eq:neq-powers-mu-leq1half}), ends the proof of Proposition \ref{prop:formulas-for-Gamma-1ptBUofCP2} in the last remaining case, for $0 < \mu \leq \frac 12$.
 \hfill $\square$

\bibliographystyle{alpha}
\bibliography{C0Seidel}

\def\cprime{$'$}
\begin{thebibliography}{AHV24}

\bibitem[AHV24]{AHV}
Marie-Claude Arnaud, Vincent Humilière, and Claude Viterbo.
\newblock Higher dimensional {B}irkhoff attractors.
\newblock {\em arXiv:2404.00804}, 2024.

\bibitem[AL17]{MR3681107}
S\'{\i}lvia Anjos and R\'{e}mi Leclercq.
\newblock Noncontractible {H}amiltonian loops in the kernel of {S}eidel's
  representation.
\newblock {\em Pacific J. Math.}, 290(2):257--272, 2017.

\bibitem[AL18]{MR3798327}
S\'{\i}lvia Anjos and R\'{e}mi Leclercq.
\newblock Seidel's morphism of toric 4-manifolds.
\newblock {\em J. Symplectic Geom.}, 16(1):1--68, 2018.

\bibitem[AM00]{MR1775741}
Miguel Abreu and Dusa McDuff.
\newblock Topology of symplectomorphism groups of rational ruled surfaces.
\newblock {\em J. Amer. Math. Soc.}, 13(4):971--1009, 2000.

\bibitem[BHS18]{MR3827210}
Lev Buhovsky, Vincent Humili\`ere, and Sobhan Seyfaddini.
\newblock A {$C^0$} counterexample to the {A}rnold conjecture.
\newblock {\em Invent. Math.}, 213(2):759--809, 2018.

\bibitem[BHS21]{MR4263685}
Lev Buhovsky, Vincent Humili\`ere, and Sobhan Seyfaddini.
\newblock The action spectrum and {$C^0$} symplectic topology.
\newblock {\em Math. Ann.}, 380(1-2):293--316, 2021.

\bibitem[EP03]{MR1979584}
Michael Entov and Leonid Polterovich.
\newblock Calabi quasimorphism and quantum homology.
\newblock {\em Int. Math. Res. Not.}, 2003(30):1635--1676, 2003.

\bibitem[Flo88a]{Floer_Morse_thry_Lagr_intersections}
Andreas Floer.
\newblock Morse theory for {L}agrangian intersections.
\newblock {\em J. Differential Geom.}, 28(3):513--547, 1988.

\bibitem[Flo88b]{Floer_unregularized_grad_flow_symp_action}
Andreas Floer.
\newblock The unregularized gradient flow of the symplectic action.
\newblock {\em Comm. Pure Appl. Math.}, 41(6):775--813, 1988.

\bibitem[Flo89a]{Floer_Symp_fixed_pts_holo_spheres}
Andreas Floer.
\newblock Symplectic fixed points and holomorphic spheres.
\newblock {\em Comm. Math. Phys.}, 120(4):575--611, 1989.

\bibitem[Flo89b]{Floer_Witten_cx_inf_dim_Morse_thry}
Andreas Floer.
\newblock Witten's complex and infinite-dimensional {M}orse theory.
\newblock {\em J. Differential Geom.}, 30(1):207--221, 1989.

\bibitem[Gro85]{MR0809718}
Mikhaïl Gromov.
\newblock Pseudo holomorphic curves in symplectic manifolds.
\newblock {\em Invent. Math.}, 82(2):307--347, 1985.

\bibitem[Hum08]{Humiliere-completion}
Vincent Humili\`ere.
\newblock On some completions of the space of {H}amiltonian maps.
\newblock {\em Bull. Soc. Math. France}, 136(3):373--404, 2008.

\bibitem[Jan21]{Jannaud1}
Alexandre Jannaud.
\newblock Dehn-{S}eidel twist, ${C}^0$ symplectic topology and barcodes.
\newblock {\em arXiv:2101.07878}, 2021.

\bibitem[Jan22]{Jannaud2}
Alexandre Jannaud.
\newblock Free subgroup of the ${C}^0$ symplectic mapping class group.
\newblock {\em arXiv:2211.05570}, 2022.

\bibitem[LC05]{MR2217051}
Patrice Le~Calvez.
\newblock Une version feuillet\'{e}e \'{e}quivariante du th\'{e}or\`eme de
  translation de {B}rouwer.
\newblock {\em Publ. Math. Inst. Hautes \'{E}tudes Sci.}, (102):1--98, 2005.

\bibitem[LC06a]{MR2275671}
Patrice Le~Calvez.
\newblock From {B}rouwer theory to the study of homeomorphisms of surfaces.
\newblock In {\em International {C}ongress of {M}athematicians. {V}ol. {III}},
  pages 77--98. Eur. Math. Soc., Z\"{u}rich, 2006.

\bibitem[LC06b]{MR2219272}
Patrice Le~Calvez.
\newblock Periodic orbits of {H}amiltonian homeomorphisms of surfaces.
\newblock {\em Duke Math. J.}, 133(1):125--184, 2006.

\bibitem[McD02]{MR1959582}
Dusa McDuff.
\newblock Geometric variants of the {H}ofer norm.
\newblock {\em J. Symplectic Geom.}, 1(2):197--252, 2002.

\bibitem[McD10]{MR2563682}
Dusa McDuff.
\newblock Monodromy in {H}amiltonian {F}loer theory.
\newblock {\em Comment. Math. Helv.}, 85(1):95--133, 2010.

\bibitem[MT06]{MR2210662}
Dusa McDuff and Susan Tolman.
\newblock Topological properties of {H}amiltonian circle actions.
\newblock {\em IMRP Int. Math. Res. Pap.}, pages 72826, 1--77, 2006.

\bibitem[Oh05a]{MR2103018}
Yong-Geun Oh.
\newblock Construction of spectral invariants of {H}amiltonian paths on closed
  symplectic manifolds.
\newblock In {\em The breadth of symplectic and {P}oisson geometry}, volume 232
  of {\em Progr. Math.}, pages 525--570. Birkh\"{a}user Boston, Boston, MA,
  2005.

\bibitem[Oh05b]{MR2181090}
Yong-Geun Oh.
\newblock Spectral invariants, analysis of the {F}loer moduli space, and
  geometry of the {H}amiltonian diffeomorphism group.
\newblock {\em Duke Math. J.}, 130(2):199--295, 2005.

\bibitem[Ost06]{ostrover}
Yaron Ostrover.
\newblock Calabi quasi-morphisms for some non-monotone symplectic manifolds.
\newblock {\em Algebr. Geom. Topol.}, 6:405--434, 2006.

\bibitem[Pau24]{Paulin}
Frédéric Paulin.
\newblock Notes de cours : Topologie algébrique.
  https://www.imo.universite-paris-saclay.fr/~frederic.paulin/notescours/cours$\_$topoalg.pdf,
  2024.

\bibitem[Sas74]{MR0346783}
Seiya Sasao.
\newblock The homotopy of {${\rm Map}\ (CP\sp{m},\,CP\sp{n})$}.
\newblock {\em J. London Math. Soc. (2)}, 8:193--197, 1974.

\bibitem[Sch00]{MR1755825}
Matthias Schwarz.
\newblock On the action spectrum for closed symplectically aspherical
  manifolds.
\newblock {\em Pacific J. Math.}, 193(2):419--461, 2000.

\bibitem[Sei97]{MR1487754}
Paul Seidel.
\newblock {$\pi_1$} of symplectic automorphism groups and invertibles in
  quantum homology rings.
\newblock {\em Geom. Funct. Anal.}, 7(6):1046--1095, 1997.

\bibitem[She22]{MR4480149}
Egor Shelukhin.
\newblock Viterbo conjecture for {Z}oll symmetric spaces.
\newblock {\em Invent. Math.}, 230(1):321--373, 2022.

\bibitem[Vit92]{MR1157321}
Claude Viterbo.
\newblock Symplectic topology as the geometry of generating functions.
\newblock {\em Math. Ann.}, 292(4):685--710, 1992.

\bibitem[Vit22]{viterbo2022supports}
Claude Viterbo.
\newblock On the supports in the {H}umili\`ere completion and
  $\gamma$-coisotropic sets.
\newblock {\em arXiv:2204.04133}, 2022.

\bibitem[Vit23]{Viterbo-Homogenization}
Claude Viterbo.
\newblock Symplectic homogenization.
\newblock {\em J. \'Ec. polytech. Math.}, 10:67--140, 2023.

\end{thebibliography}

\end{document}